\documentclass{article}
\usepackage{fancyhdr}
\usepackage{latexsym}
\usepackage{mathrsfs}
\usepackage{cancel}
\usepackage{color}
\usepackage{multirow}
\usepackage{eso-pic}
\usepackage{dsfont}
\usepackage{amssymb}
\usepackage{graphicx}
\usepackage{setspace}
\usepackage{geometry}
\usepackage{amsthm}
\usepackage{hyperref}
\usepackage[intlimits]{amsmath}
\usepackage[capitalize]{cleveref} 
\usepackage{tikz-cd}
\usepackage{amssymb,amsfonts}
\usepackage[all,arc]{xy}
\usepackage{enumerate}
\usepackage{mathrsfs}
\usepackage{tikz}
\usepackage{bbm}
\usepackage{graphicx}
\newcommand{\wt}{\widetilde}
\newcommand{\wh}{\widehat}

\newcommand{\End}{\mathrm{End}}

\newcommand{\Hom}{\mathrm{Hom}}

\newcommand{\Mod}{\mathrm{-Mod}}

\newcommand{\Gr}{\mathrm{Gr}}

\newcommand{\lp}{\left(}
\newcommand{\rp}{\right)}

\newcommand{\QCoh}{\mathrm{QCoh}}

\newcommand{\MF}{\mathrm{MF}}

\newcommand{\Rep}{\mathrm{Rep}}

\newcommand{\lbb}{[\![}
\newcommand{\rbb}{]\!]}

\newcommand{\pd}{{\partial}}

\newcommand{\HC}{\mathrm{HC}}

\newcommand{\biMod}{\mathrm{-BiMod}}
\newcommand{\Ce}{\mathrm{CE}}
\newcommand{\Dist}{\mathrm{Dist}}

\newcommand{\C}{\mathbb C}

\newcommand{\N}{\mathbb N}

\newcommand{\fg}{\mathfrak{g}}
\newcommand{\fh}{\mathfrak{h}}
\newcommand{\fk}{\mathfrak{k}}

\newcommand{\fa}{\mathfrak{a}}

\newcommand{\CA}{{\mathcal A}}

\newcommand{\CC}{{\mathcal C}}
\newcommand{\CD}{{\mathcal D}}

\newcommand{\CF}{{\mathcal F}}

\newcommand{\CM}{{\mathcal M}}
\newcommand{\CN}{{\mathcal N}}
\newcommand{\CO}{{\mathcal O}}

\newcommand{\CR}{{\mathcal R}}

\newcommand{\CW}{{\mathcal W}}  
\newcommand{\CX}{{\mathcal X}}
\newcommand{\CY}{{\mathcal Y}}
\newcommand{\CZ}{{\mathcal Z}}

\definecolor{cadmiumorange}{rgb}{0.93, 0.53, 0.18}

\newcommand{\be}{\begin{equation}}
\newcommand{\ee}{\end{equation}}

\newcommand{\btik}{\begin{tikzcd}}
\newcommand{\etik}{\end{tikzcd}}

\newtheorem{Def}{Definition}[section]
\newtheorem{Thm}[Def]{Theorem}
\newtheorem{Prop}[Def]{Proposition}
\newtheorem{Cor}[Def]{Corollary}
\newtheorem{Lem}[Def]{Lemma}

\theoremstyle{definition}
\newtheorem{Exp}[Def]{Example}
\newtheorem{Rem}[Def]{Remark}

\numberwithin{equation}{section}

\title{D-convolution categories and Hopf algebras}
\author{Wenjun Niu}
\date{\today}

\begin{document}

\maketitle

\begin{abstract}
   For a smooth affine algebraic group $G$, one can attach various D-module categories to it that admit convolution monoidal structure. We consider the derived category of D-modules on $G$, the stack $G/G_{ad}$ and the category of Harish-Chandra bimodules. Combining the work of Beilinson-Drinfeld on D-modules and Hecke patterns with the recent work of the author with Dimofte and Py, we show that each of the above categories (more precisely the equivariant version) is monoidal equivalent to a localization of the DG category of modules of a graded Hopf algebra. As a consequence, we give an explicit braided monoidal structure to the derived category of D-modules on $G/G_{ad}$, which when restricted to the heart, recovers the braiding of Bezrukavnikov-Finkelberg-Ostrik. 
   \end{abstract}

\tableofcontents

\section{Introduction}

Let $G$ be a smooth affine algebraic group. There are three triangulated monoidal categories
\be
\CD (G), \qquad \CD (G/G_{ad}), \qquad D(\HC (G)),
\ee
the category of D-modules on $G$, on $G/G_{ad}$ (the adjoint quotient of $G$ by itself) and the category of Harish-Chandra bimodules, respectively. These categories are interesting and important in geometric representation theory. Although the monoidal structures can be elegantly defined geometrically, their computations are usually difficult. Moreover, although it is known that the category $\CD (G/G_{ad})$ has the structure of a braided tensor category \cite{boyarchenko2014character}, and is equivalent to some notion of the center of the other two monoidal categories \cite{bezrukavnikov2023equivariant}, the braiding is still difficult to compute in general. 

The objective of this paper is to translate these geometric monoidal categories in Hopf-algebraic terms, which makes the monoidal structure more transparent. To do so, we will consider the equivariant (or asymptotic) version of the above categories
\be
\CD_\hbar (G)^{\C^\times}, \qquad \CD_\hbar (G/G_{ad})^{\C^\times}, \qquad D(\HC_\hbar (G)^{\C^\times}),
\ee
in which we promote filtered algebras into their Reez algebras and consider $\C^\times$-equivariant modules. We apply the work of Beilinson-Drinfeld on D-modules and Hecke patterns \cite{BDhecke} in this equivariant setting, as well as recent works on Hopf algebras in gauge theories \cite{DN24, NPShifted}, to construct three (cohomologically) graded topological Hopf algebras
\be
\CA_G,\qquad A_{G/G_{ad}}, \qquad H_G.
\ee
We show that these Hopf algebras control monoidal structures of the geometric categories above. 

\begin{Thm}\label{Thm:DHequivintro}
There are equivalences of triangulated monoidal categories
\be
\begin{aligned}
\CD_\hbar (G)^{\C^\times}\simeq D_{qs}&(\CA_G\Mod^{\C^\times}), \qquad D(\HC_\hbar (G)^{\C^\times})\simeq D_{qs}(H_G\Mod^{\C^\times})\\
&\CD_\hbar (G/G_{ad})^{\C^\times}\simeq D_{qs}(A_{G/G_{ad}}\Mod^{\C^\times}),
\end{aligned}
\ee
where $D_{qs}$ means certain localization of the DG category of DG modules (to be recalled in Section \ref{subsec:KoszulDO}). 
\end{Thm}

The equivalences above at the level of triangulated categories are established in \cite[Section 7]{BDhecke}, which are examples of Koszul duality for D-modules. We will construct the Hopf structures in Section \ref{sec:Hopfalgebra}, and prove this statement in Theorem \ref{Thm:HCmonoidal}, Corollary \ref{Cor:coprodCAG} and Theorem \ref{Thm:cDOmonoidal}. 

\begin{Rem}

The topology on the algebras above come from the pro-finite topology of the linear dual of $\CO_G$, which is an infinite-dimensional vector space filtered by finite-dimensional  subspaces. See Appendix \ref{app:double} for details. 

\end{Rem}

We then turn to braided monoidal structure of $\CD_\hbar (G/G_{ad})^{\C^\times}$. This was first constructed in \cite[Appendix B]{boyarchenko2014character} using functoriality of $D$-modules\footnote{Although in this work the authors worked with the category of constructible sheaves, their construction works equally well for the category $\CD_\hbar$, as all the functors needed for the construction have counterparts in the asymptotic setting.}. In the $\infty$-category setting, this was constructed in the work of \cite{beraldo2017loop}. However, neither of these works gave an explicit description of the braiding. We apply Theorem \ref{Thm:DHequivintro} and show that a braiding exists on $A_{G/G_{ad}}$ in terms of the braiding of a double.

\begin{Thm}\label{Thm:doubleintro}
The topological quantum double of $H_G$ and $\CA_G$ are twist equivalent, and the double $D(\CA_G)$ is a central extension of $A_{G/G_{ad}}$ by a primitive element $\delta^*$ in degree $-1$. The $R$-matrix of $D(\CA_G)$ descends to an $R$-matrix in the quotient $A_{G/G_{ad}}=D(\CA_G)/(\delta_*)$. 

\end{Thm}

\begin{Rem}

A ribbon element $\theta$ exists for $A_{G/G_{ad}}$ (Remark \ref{Rem:twist} and Proposition \ref{Prop:twist}). In particular, the category of finite-dimensional modules of $A_{G/G_{ad}}$ has the structure of a pivotal braided monoidal category. 

\end{Rem}

This theorem is proven in Section \ref{sec:double}, Theorem \ref{Thm:double}. The fact that the double of $H_G$ and $\CA_G$ are twist equivalent is expected from the work of \cite{beraldo2017loop, BGngo}. It follows from the calculation of the $R$-matrix that when restricted to the heart of the standard t structure of D-module categories, the braiding of Theorem \ref{Thm:doubleintro} recovers the braiding of \cite{BFOcharacter} (see Section \ref{subsubsec:CSH}). We expect that the full braiding can be comparable to the one defined in \cite{boyarchenko2014character}. We now give some explanation for the construction of the Hopf algebras in terms of topological quantum field theories.

\subsection{Hopf algebras from 3 dimensional TQFTs}

The category $\HC_\hbar (G)$ is featured in the study of geometric Langlands correspondence via the following geometric Satake equivalence \cite{bezrukavnikov2008equivariant}
\be\label{eq:qsatake}
\CD (\C^\times\!\!\ltimes \!\! G(\CO)\setminus\!\Gr_G)\simeq D(\HC_\hbar ({}^LG)),
\ee
where the LHS is the derived category of $\C^\times\!\ltimes G(\CO)$-equivariant $D$-modules on $\Gr_G$. Here $\C^\times$ acts by loop rotation. Turning-off the $\C^\times$ action on the LHS corresponds to $\hbar\to 0$ on the RHS. Taking the heart of the perverse t-structure on the LHS, one obtains the more familiar statement
\be\label{eq:satake}
\mathrm{Perv}_{G(\CO)}(\Gr_G)\simeq \mathrm{Rep}({}^L G). 
\ee
Both the geometric Langlands correspondence and the above categories can be given a physical interpretation in terms of topological twists of 4 dimensional super Yang-Mills theories \cite{kapustin2007electric}. The above categories are the category of \textit{line operators} of the gauge theory, namely observables that are supported along a line. The equivalence is a consequence of the physical electro-magnetic (S) duality. 

Dimensional reduction of these theories along a circle are 3 dimensional TQFTs, which are topological twists of 3d $\CN=4$ (actually, $\CN=8$) gauge theories. The general cigar reduction scheme (explained well in \cite{Butsonequiv}) makes the above categories into the category of line operators on a boundary condition of the 3d TQFT (the cigar boundary). 3 dimensional $\CN=4$ gauge theories also enjoy their version of electro-magnetic duality called \textit{mirror symmetry} \cite{intriligator1996mirror}, which is related to symplectic duality \cite{webster20233}. The dimensional reduction of the above 4d theories (for $G$ and ${}^LG$) form a dual pair 3d theories. 

Mirror symmetry is not our focus, so we will stop the discussion here, and turn to explaining the origin of the Hopf algebras. For this we focus on the category on the RHS of \eqref{eq:qsatake} (physically, the B side), and use $G$ instead of ${}^LG$. Since there is a 3 dimensional TQFT, we expect, from general physical non-sense, a braided tensor category $\CC$ associated to it. Since the category on the RHS is the category associated to a boundary condition, we expect that $\CC$ is equal to, or at least maps to, the Drinfeld center of $D(\HC_\hbar (G))$. 

This TQFT is somewhat understood from geometric representation theory perspective. The category $\CC$ is expected to be the category of D-modules on the character stack $G/G_{ad}$, which, using the work of \cite{ben2010integral, ben2012loop, benzvi2017characterfieldtheoryhomology}, can be written as
\be
D\lp \QCoh \lp \mathrm{Maps}(S^1_L\times S^1_R, BG)^{S^1_L}\rp\rp.
\ee
Here $S^1_L, S^1_R$ are the topological loop spaces and the supscript rotates one of the two copies of $S^1$ (the left one). The braided tensor structure comes from natural cobordisms on $S^1_R$ that is not being rotated. Moreover, this category indeed agrees with the Drinfeld center of $D(\HC_\hbar (G))$, proven in \cite{beraldo2017loop}.  This description of the category is elegant, and allows an easy description of state space on $\Sigma$, as functions on $\mathrm{Maps}(S^1_L\times \Sigma, BG)^{S^1_L}$. However, the braided monoidal structure is described in terms of $\infty$-categorical nonsense. This makes it difficult to extract, for example, knot invariants.

Fortunately, we can completely control the braided tensor structure of $\CC$ (at least on physical grounds) using quasi-triangular Hopf algebras. This was done in \cite{DN24}. The idea is that the TQFT in question has a pair of \textit{transverse} and \textit{complete} boundary conditions (different from the cigar boundary condition). They are simply called the Dirichlet and Neumann boundary conditions. When this happens, it was argued in \textit{op.cit.}, that sandwiching the gauge theory by the two boundary conditions gives a fiber functor for the monoidal category $\CC$ as well as the monoidal categories associated to the boundaries. For the TQFT in question, the Hopf algebra $\CA_{G}$ and its double $D(\CA_{G})$ are precisely the symmetry algebras of the fiber functors, as in \cite[Section 9]{DN24}. Furthermore, the equivalence in Theorem \ref{Thm:DHequivintro} gives a consistency check of the proposal of \cite{DN24} with the geometric approach of \cite{ben2010integral, ben2012loop}.\footnote{We comment that this quotient by $\delta^*$ is a mild modification, and comes from our choice of treating the  differential as an element in the algebra, similar to \cite{ben2012loop}.}

However, this does not explain the Hopf structure of $H_G$, which we explain in terms of 1-shifted Lie bialgebras.

\subsection{1-shifted metric Lie algebras}

A Lie bialgebra is the classical limit of a Hopf algebra, and the double of a Lie bialgebra is the classical limit of a quasi-triangular Hopf algebra. Given a Lie bialgebra, its quantization problem was solved in \cite{etingof1996quantization}, which is called the canonical quantization of the Lie bialgebra. In \cite{DN24}, their quantization was given a physical interpretation as constructing transverse and complete boundary conditions for a 3d Chern-Simons theory. 

Unfortunately, besides the work of \cite{pimenov}, there hasn't been much study of Lie bialgebras in cohomologically shifted settings. In \cite{NPShifted}, we studied 1-shifted Lie bialgebras and their quantizations. A 1-shifted Lie bialgebra is a graded Lie algebra $\fa$ with a cobracket $\delta$ of degree $1$, satisfying a graded cocycle condition similar to ordinary Lie bialgebras. It was show in \textit{loc.cit} that such an object admits a 1-shifted dual $\fa^*[-1]$ and double $D(\fa)$, and that the cobracket on the double is coboundary. The double $D(\fa)$ is a 1-shifted metric Lie algebra. We will give a quick review of these in Section \ref{subsec:HopfHC}.  

Just like the Chern-Simons theory set-up, the triple $(D(\fh), \fh, \fh^*[-1])$ gives rise to a classical bulk-boundary 2d TQFT (which was explained in \cite[Section 1.1]{NPShifted}). Based on this, a quantization was proposed in \textit{op.cit}. More specifically, denote by $U_\delta (\fh)$ the algebra $\C[\delta]\ltimes U(\fh)$, and similarly define $U_\delta (\fh^*[-1])$ and $U_\delta (D(\fh))$, then it was shown that the following is true.

\begin{Thm}\label{Thm:1shiftedintro}
$U_\delta (D(\fh))$ is a Hopf algebra, and $U_\delta (\fh), U_\delta (\fh^*[-1])$ are module algebras of $U_\delta (D(\fh))$.  Assuming $D(\fh)$ does not have degree $2$ component, then the statements remain true after quotienting out by $\delta^2=0$. 

\end{Thm}

The category $U_\delta (D(\fh))\Mod$ can be thought of as the category of interfaces of the 2d TQFT and $U_\delta (\fh)\Mod$ the category of boundary conditions. In particular, this set-up gives a monoidal functor
\be
U_\delta (D(\fh))\Mod\to \End (U_\delta (\fh)\Mod).
\ee
In \cite{NPShifted}, Theorem \ref{Thm:1shiftedintro} was stated in terms of curved differential graded algebras since $\delta$ was treated as a differential there. In this paper we treat it as an extra element. 

Coming back to the original 3d gauge theory, the cigar boundary condition of the 3d TQFT is precisely such a 2d TQFT \cite{elliott2022taxonomy}. Here the Lie algebra $\fh=\fg$ with trivial co-bracket. Therefore, we expect that the Hopf algebra $U_\delta (D(\fh))/\delta^2$ \footnote{This quotient by $\delta^2$ is because $\delta$ is not really an element in the algebra, but the differential. We choose to treat the differential as an element so as to obtain Hopf structure.} can be used to express the monoidal category associated to the cigar boundary condition. This is precisely the construction of $H_G$, except one replaces the Lie algebra $\fg$ by the algebraic group $G$ to account for the geometry of $G$. Theorem \ref{Thm:DHequivintro} and Theorem \ref{Thm:doubleintro} confirms the physical prediction. 

From this perspective, it is interesting that the cigar boundary supporting the category $D(\HC_\hbar (G)^{\C^\times})$ admits a transverse boundary, supporting $H_G^*\Mod^{\C^\times}$. Moreover, the algebra $H_G^*$ is in some sense a ``quantization" of the 1-shifted symplectic structure on $T[1]G$ (see Section \ref{subsubsec:ShiftedPoisson}). The physical meaning of this is unclear to the author at the moment.

\subsection{Organizations and conventions}

\subsubsection*{Organizations of the paper}

\begin{itemize}

\item In Section \ref{sec:Dcat}, we recall the work of Beilinson-Drinfeld on D-module categories and Hecke patterns. In particular, we recall their work on Koszul dualities. We extend their Koszul duality (extremely) slightly in the case of Harish-Chandra modules to a partial Koszul duality. 

\item In Section \ref{sec:Hopfalgebra}, we introduce the Hopf algebras $H_G, \CA_G$ and $A_{G/G_{ad}}$, and prove Theorem \ref{Thm:DHequivintro}. The proof for $\CA_G$ and $A_{G/G_{ad}}$ follows from the functoriality of the Koszul duality of Beilinson-Drinfeld. 

\item In Section \ref{sec:double}, we compute the double of $\CA_G$ and $H_G$, and prove Theorem \ref{Thm:doubleintro}. 

\item In Appendix \ref{app:double}, we consider Drinfeld double for an infinite-dimensional Hopf algebra. This provides the technical details for the rigorous construction of the algebras in the previous sections. 

\end{itemize}

\subsection*{Conventions}

In this paper, by ``grading" it will always mean cohomological grading, whereas grading coming from any external $\C^\times$ action will simply be refered to as ``weights". Even though the algebras involved will be generated in various cohomological degrees (namely non-classical), one can perform a shearing using the $\C^\times$ grading so that they become classical. This shearing introduces extra signs in the coproduct formulas according to the Koszul sign rule, and we prefer not to deal with the signs. In particular, Koszul sign rule is always assumed, and all Hopf algebras are Hopf algebras in the category of chain complexes. 

All structures are considered linear over $\C$. Tensor product over $\C$ is denoted by $\otimes$, and the completed tensor product is denoted by $\wh\otimes$, if there is a pro-finite topology. If an algebra $A$ admits a topology, then $A\Mod$ is always the category of smooth modules. For a graded algebra, the category $A\Mod$ is the category of ordinary graded modules (an abelian category), and $C(A\Mod)$ is the DG category of chain complexes of modules. The derived category is denoted by $D(A\Mod)$. 

\subsection*{Acknowledgements}

Many thanks are due to Tudor Dimofte and Victor Py for collaborations that led to this work, to Tom Gannon and  Lukas M\"uller for many helpful discussions, and to my friend Don Manuel for the continued support. My research is supported by Perimeter Institute for Theoretical Physics. Research at Perimeter Institute is supported in part by the Government of Canada through the Department of Innovation, Science and Economic Development Canada and by the Province of Ontario through the Ministry of Colleges and Universities.

\section{D-convolution categories attached to algebraic groups}\label{sec:Dcat}

\subsection{D-modules on stacks and Harish-Chandra modules}

\subsubsection{Generalities on D-modules}

For a smooth affine scheme $X$, we denote by $D_X$ the sheaf of differential operators on $X$ and $\CM(X)$ the abelian cartegory of (right) D-modules on $X$, namely $\CM (X)=D_X^{op}\Mod$. Its DG category of chain complexes is denoted by $C(\CM (X))$ and its derived category is denoted by $\CD(X)$. Let $\CX$ be a smooth algebraic stack. For the applications in this paper, one may assume $\CX$ is of the form $X/G$ for some smooth affine scheme $X$ modulo the action of a smooth affine algebraic group $G$. In \cite[Section 7.5]{BDhecke}, Beilison-Drinfeld defined the derived category of $D$-modules on such $\CX$ in the following way. Choose $\{U_*\}_{*\in \N}$ a hypercovering of $\CX$ such that each $U_n$ is a smooth affine scheme (referred to as an \textit{admissible} hypercovering in \textit{loc.cit.}), and consider the corresponding simplicial set of abelian and derived categories
\be\label{eq:simplicial}
\{U_*\}\rightsquigarrow \CM (U_*), ~ \CD(U_*). 
\ee
The following was proven in \cite[Proposition 7.5.2]{BDhecke}.

\begin{Prop}

    Both simplicial categories in \eqref{eq:simplicial} admit totalizations, as abelian and triangulated categories, respectively. Moreover, the totalizations do not depend on the choice of admissible coverings $\{U_*\}$. 
    
\end{Prop}

Based on this proposition, one can define categories $\CM (\CX)$ and $\CD (\CX)$. The heart of the latter is the former abelian category, but the derived category of the former is usually not equal to the latter, unless $\CX$ is very nice (for example, a Deligne-Mumford stack). The following is the main example considered in this paper.

\begin{Exp}\label{Exp:heartvsderived}

    Let $G$ be a smooth affine algebraic group. Then the category $\CD (G)$ is equivalent to the derived category of modules of the algebra $D_G^{op}$, which can be identified with $U(\fg)\ltimes \CO_G$, where $U(\fg)$ acts on $\CO_G$ via left-invariant vector fields. On the other hand, consider the stack $\CX=G/G_{ad}$ where the quotient is by adjoint action. The category $\CD (G/G_{ad})$ is the triangulated category of strongly equivariant D modules on $G$, while category $\CM (G/G_{ad})$ is the category of modules of $D_G^{op}$ where the action of $\fg_{ad}$ by conjugation integrates to an action of $G_{ad}$. In the triangulated category, strong-equivariance is a data rather than a condition.  
    
\end{Exp}

Let $f: \CX\to \CY$ be a  quasi-compact morphism, then one can define $f_*: \CD(\CX)\to \CD (\CY)$ by choosing an admissible hypercover $\{U_n\}$ of $\CX$ and $\{V_n\}$ of $\CY$ and a compatible morphism $\{f_n\}$, and define $f_*$ to be the morphism $\{f_{n, *}\}$ (the derived functor) on hypercoverings. Similarly, if $f: \CX\to \CY$ is a smooth morphism, then the pullback functor $\{f^!_n\}$ (underived) defined on an admissible hypercover is exact, and induces a functor $f^!: \CD (\CY)\to \CD (\CX)$. If $f$ is also quasi-compact, then $(f^!, f_*)$ forms an adjoint pair. The formation of pushforward and pullback functors respect compositions. 

This adjoint pair satisfy the base-change property. Namely, giving a base-change diagram of quasi-compact morphisms
\be
\btik
\CW\rar{\wt f}\dar{\wt g} & \CY\dar{g}\\
\CX\rar{f} & \CZ 
\etik
\ee
and suppose that $g$ and $\wt g$ are smooth, then the canonical adjoint natural transformation $g^!f_*\to \wt f_* \wt g^!$ is a natural isomorphism. 

Let $G$ be an affine algebraic group with a subgroup $K\subseteq G$, and consider the Harish-Chandra pair $(\fg, K)$. The abelian category of Harish-Chandra modules $\HC (\fg, K)$ is the abelian category of modules of $U(\fg)$ where the action of $\mathfrak k=\mathrm{Lie}(K)$ integrates to an action of $K$. Its derived category is denoted by $D(\HC (\fg, K))$. This category is closely related to the category of D-modules in the following way. Consider the generalized flag variety $K\setminus G$, which admits a right action of $K$. The category of D-modules on the double quotient $K\setminus G/K$ is a monoidal category, called the Hecke category. A variant of this category is the category of weakly $K$-equivariant D-modules on $K\setminus G$, which we denote by $\CD(K\setminus G)^K$. Taking global sections provides an equivalence of triangulated categories
\be
\btik
\CD(K\setminus G)^K\rar{\Gamma} & D(\HC (\fg, K)).
\etik
\ee
Therefore the Harish-Chandra category admits a geometric description, which gives rise to the canonical action of $\CD(K\setminus G/K)$ on $D(\HC (\fg, K))$. 

More generally, if $\CX$ is a smooth algebraic stack acted on by $G$, then there is a pair of adjoint functors
\be
\btik
D(\HC(\fg, K))\arrow[r, shift left, "\Delta"] &\arrow[l, shift left, "\Gamma (-)"] \CD(\CX/K)
\etik
\ee
Here $\Gamma(-)$ is the global section functor and $\Delta$ is the so-called ``localization functor", which is a quantization of the moment map. The above pair are functors of module categories of the Hecke category. The functor $\Delta$ is one of the main ingredients in Beilinson-Drinfeld's approach to the geometric Langlands correspondence. 

\subsubsection{Hecke categories}

We have already mentioned the Hecke category $\CD(K\setminus G/K)$. The monoidal structure is defined by pull-push along the following diagram
\be
\btik
& K\setminus G\times_K G/ K\arrow[dr]\arrow[dl]& \\
K\setminus G/K\times K\setminus G/K & & K\setminus G/K
\etik
\ee
We can consider 2 specializations. The first one is when $K=1$ the trivial group, in which case the Hecke category is simply $\CD (G)$. We can also consider the case when $G=K\times K$ and $K=K_\Delta$ the diagonal group, then the stack $K\setminus G/K$ is naturally identified with $K/K_{ad}$, the adjoint quotient of $K$ by $K$. This stack is called the character stack of $K$, and can be identified with the topological loop space of $*/K$, as in \cite{ben2012loop}. From the above discussions, the category $\CD (K/K_{ad})$ admits a monoidal structure, via the identification with $K\setminus G/K$. Moreover, it is braided monoidal \cite{boyarchenko2014character}. 

On the other hand, the category $\HC (\fg, K)$ does not usually admit a monoidal structure. However, we can again restrict to the case when $G=K\times K$ with $K=K_\Delta$. In this case, $\HC (\fg, K)$ is the category of Harish-Chandra bimodules of $\mathfrak k$, and does admit a monoidal structure, via the forgetful functor $\HC (\fg, K)\to U(\mathfrak k)\biMod$. We denote this category simply by $\HC (K)$. The action of $\CD (K/K_{ad})$ on $D(\HC (K) )$ is given by a monoidal functor $\CD (K/K_{ad})\to D(\HC (K) )$, and upgrades to a functor to the center. The statement for the abelian category is proven in \cite{BFOcharacter}. The statement for the derived categories follows from the corresponding statement for the $\infty$-categories, which was proven in \cite{beraldo2017loop}. We summarize this into the following theorem.

\begin{Thm}[\cite{BFOcharacter} \& \cite{beraldo2017loop}]\label{Thm:center=character}
    There are equivalences of abelian and $\infty$-categories respectively
    \be
\CZ (\HC (K))\simeq \CM (K/K_{ad}),\qquad \CZ (D(\HC (K)))\simeq \CD (K/K_{ad})). 
    \ee
\end{Thm}

We make two remarks. The first remark is that the abelian category statement of \cite{BFOcharacter} does not imply the derived category statement, as in Example \ref{Exp:heartvsderived}. This is again due to the fact that strong-equivariance in the abelian category is a condition, but in the derived category is an extra data.  The second remark is that the braided monoidal structures defined by \cite{boyarchenko2014character} and \cite{beraldo2017loop} are not explicit, except in the case of unipotent groups. Our main result is that one can give the underlying triangulated category of the RHS a braided monoidal structure using an explicit quasi-triangular Hopf algebra. To be able to do so, we apply Koszul duality in the context of D-modules. We review this in the next section, again following \cite{BDhecke}.

\subsection{Koszul dualities in D-module categories}\label{subsec:KoszulDO}

\subsubsection{D-modules and de-Rham complexes}

Let $X$ be a smooth affine scheme, and let $\Omega_X$ be the ring of differential forms on $X$. This graded algebra can be naturally identified with functions on $T[1]X$, the $(-1)$-shifted tangent space of $X$. This algebra admits a square-zero differential $d_{dR}$, the de-Rham differential. Let $\Omega_X^{dR}$ be the DG algebra given by $(\Omega_X, d_{dR})$, which is the algebraic version of the de-Rham complex of $X$. There is a pair of adjoint functors between the DG categories of chain complexes
\be\label{eq:DOadjoint}
\btik
C(\CM (X))\arrow[r, shift left, "\Omega"] & C(\Omega_X^{dR}\Mod)\arrow[l, shift left, "D"]
\etik
\ee
where $\Omega$ sends a D-module $M$ to the de-Rham complex corresponding to $M$. One can show that the derived endomorphism algebra $\mathrm{REnd} (\CO_X)$ is quasi-isomorphic to $\Omega_X^{dR}$, and the above adjunction is simply a tensor-Hom adjunction (hence Koszul duality). However, descending to the derived categories, this is usually not an equivalence. In \cite[Section 7.2.5]{BDhecke}, this was fixed by inverting morphisms of $\Omega_X^{dR}$ modules that are quasi-isomorphisms after applying $D$ (therefore not the usual derived category). Another way to deal with this is by passing to a $\C^\times$ equivariant category, and considering bounded derived categories, as in \cite{beilinson1996koszul}. This has the advantage that the Koszul dual is a ordinary Hopf algebra rather than a DG Hopf algebra (essentially treating the de-Rham differential as a generator as in \cite[Section 5]{ben2012loop}), but the drawback in our setting is that the bounded derived category does not admit a monoidal structure, as the map $G\times G\to G$ is not proper. We therefore combine the two approaches: we consider $\C^\times$ equivariant categories that are not neccesarily bounded, and invert quasi-isomorphisms in the same way as \cite[Section 7.2.5]{BDhecke}.

For $X$ a smooth affine scheme, the algebra $D_X$ is a filtered algebra, whose associated graded is naturally identified with functions on $T^*X$. We consider the Rees algebra $D_X^\hbar$ associated to $D_X$, which is an algebra with a $\C^\times$ action, and the category of $\C^\times$-equivariant modules $\CM_\hbar (X)^{\C^\times}$. Locally, the action of $\C^\times$ has degree $-1$ on the tangent vectors and $\hbar$. The formalism of the previous section works equally well in this setting, allowing us to define the derived category $\CD_\hbar (\CX)^{\C^\times}$ for a smooth stack $\CX$, whose heart is the abelian category $\CM_\hbar (\CX)^{\C^\times}$. 

On the other hand, the algebra $\Omega_X$ carries a natural $\C^\times$ action, whose weights coincide with the grading. Denote by $\CA_X$ (following the notation of \cite{ben2012loop}) the algebra
\be
\CA_X:= \Omega_X[\delta_{dR}]/(\delta_{dR}^2, [\delta_{dR}, \omega]=d_{dR}\omega).
\ee
Namely instead of considering the DG algebra $\Omega^{dR}_X$, we add a generator $\delta_{dR}$ whose commutator induces the de-Rham differential. This algebra admits a $\C^\times$ action where $\delta_{dR}$ has weight $1$. We consider the category $\CA_X\Mod^{\C^\times}$ the category of $\C^\times$ equivariant modules. Given a smooth algebraic stack $\CX$ with a smooth cover $U$, the sheaf of algebras $\CA_{U}$ descends to a sheaf of algebras $\CA_\CX$, which is still equivariant with respect to the $\C^\times$ action. One can define, similar to $\CD (\CX)$, the category $D(\CA_\CX\Mod^{\C^\times})$. 

The functors in equation \eqref{eq:DOadjoint} works equally well in this equivariant setting, providing, for each smooth affine variety $X$, a pair of adjoint functors
\be\label{eq:DOajunch}
\btik
C(\CM_\hbar (X)^{\C^\times})\arrow[r, shift left, "\Omega_\hbar"] & C(\CA_X\Mod^{\C^\times})\arrow[l, shift left, "D_\hbar"]
\etik
\ee
This pair is induced by the $\C^\times$-equivariant version of the de-Rham complex. More explicitly, consider the object
\be\label{eq:DRhbar}
DR_{\hbar ,X}:=D_X^{\hbar}\otimes_{\CO_X} \CA_X, \qquad d=\sum_i (v_i)_L\otimes (\omega^i)_L+(\hbar)_L\otimes (\delta_{dR})_L,
\ee
where $\{v_i\}$ is a (local) basis for $T_X$ and $\{\omega^i\}$ a dual basis for $\Omega^1(X)$, and the subscrip $L$ means left multiplications. Note the tensor is over $\CO_X$ as left modules. This is a $D_X^{\hbar, op}-\CA_X$ bimodule, and the functors in equation \eqref{eq:DOajunch} are given by
\be
\Omega_\hbar (M):=\Hom_{D_X^\hbar}(DR_{\hbar, X}, M), \qquad D_\hbar (N):= DR_{\hbar, X}\otimes_{\CA_X} N. 
\ee
We call a morphism in $C(\CA_X\Mod^{\C^\times})$ a $D$-quasiisomorphism if its image under $D_\hbar$ is a quasi-isomorphism. Denote by $D_{qs} (\CA_X\Mod^{\C^\times})$ the localization of $K(\CA_X\Mod^{\C^\times})$ (the homotopy category) with respect to $D$-quasiisomorphisms. By repeating the proof of \cite[Lemma 7.2.4 \& Section 7.6.11]{BDhecke} in the $\C^\times$-equivariant setting, one arrives at the following result. 

\begin{Prop}\label{Prop:DOKoszul}
    There is an equivalence of triangulated categories
    \be
\CD_\hbar (X)^{\C^\times}\simeq D_{qs}(\CA_X\Mod^{\C^\times}).
    \ee
    Moreover, the notion of $D$-quasiisomorphism satisfies descent under smooth topology. Consequently, for any smooth stack $\CX$, there is an equivalence of triangulated categories
    \be\label{eq:DOequiv}
\CD_\hbar (\CX)^{\C^\times}\simeq D_{qs}(\CA_\CX\Mod^{\C^\times}).
    \ee
    This equivalence is functorial with respect to pushforward functors and smooth pullback functors. 
\end{Prop}

\begin{proof}

    This is a repeat of the proof of \cite[Lemma 7.2.4]{BDhecke}. Indeed, one shows that for any $M\in C(\CM_\hbar (X)^{\C^\times})$, the canonical adjunction $D\Omega M\to M$ is a quasi-isomorphism. Therefore, the adjunction $N\to \Omega DN$ is a $D$-quasiisomorphism. The second part of the statement follows from \cite[Section 7.6.11]{BDhecke}, again applied in the $\C^\times$-equivariant setting. 
    
\end{proof}

\begin{Rem}
A $D$-quasiisomorphism is a quasi-isomorphism, but the inverse is not true. A quasi-isomorphism is a $D$-quasiisomorphism when both the domain and co-domain are bounded complexes of finitely generated modules. The category on the RHS can be identified with the category of injective complexes of  \cite{krause05} or the coderived category of \cite{positselski2011two}. In modern terms \cite{Gaiindcoh}, this is the homotopy category of the category of ind-coherent modules. 
\end{Rem}

In the case when $X$ is a smooth affine variety, the algebra $\CA_X$ and its category of modules are easy to describe. Suppose now that $G$ is an algebraic group acting on $X$, then  \cite[7.6.11]{BDhecke} described the category $C(\CA_{X/G}\Mod^{\C^\times})$ explicitly. An object in this category is a complex of $\CA_X$ modules $M^*$, together with a compatible action of the algebraic group $G$, as well as $G$-equivariant maps $\iota_b: M^*\to M^{*-1}$ for $b\in\fg$ ($G$ acts by conjugation on $\fg$), such that
\be
\{\iota_b, \delta_{dR}\}=\mathrm{Lie}_b,\qquad \{\iota_b, \omega\}=(\xi_b, \omega), ~\forall \omega\in \Omega^1(X),
\ee
Here $\mathrm{Lie}_b$ denotes the action of $b$ on $M^*$ induced by the algebraic action of $G$, and $(\xi_b, \omega)$ is the natural pairing between tangent vectors and one-forms. Note here that $\xi_b$ is the vector $\xi_b (f)(x)=\pd_t\vert_{t=0} f(e^{-tb}x)$ that induces the left action of $G$ on $X$. Another way to say this is that an object in $C(\CA_{X/G}\Mod^{\C^\times})$ is a module of $\CA_X$ together with a compatible action of $G\ltimes \fg[1]=T[1]G$.

Let us consider our main examples. 

\begin{Exp}\label{exp:G}

    Consider $\CX=G$, and identify $\Omega_G$ with the algebra of functions on $T[1]G$. We can identify $T[1]G$ with $G\times \fg[1]$ via left-invariant vector fields, and so $\Omega_G=\CO_G\otimes \bigwedge {}^*\fg^*[-1]$. Let $\delta_{dR}$ be the auxiliary variable in $\CA_G$, then the commutation relation of $\CA_G$ reads
    \be
\{\delta_{dR}, f\}=\sum \xi_{x_i}^R (f)c^i, \qquad \delta_{dR}c^i=-\frac{1}{2}\sum f^i_{jk} c^jc^k. 
    \ee
 Here $\xi_{x_i}^R(f)$ is the action of $\fg$ on $\CO_G$ by left-invariant vector fields. Under the equivalence of equation \eqref{eq:DOequiv}, the monoidal structure of $\CD_\hbar (G)^{\C^\times}$ corresponds to the group structure of $T[1]G=G\ltimes \fg[1]$, the extension of $G$ by the representation $\fg$ in degree $-1$. In the next section, we will write the corresponding coproduct more explicitly. 
    
\end{Exp}

\begin{Exp}\label{exp:G/G}
    Consider $\CX=G/G_{ad}$. An object in $C(\CA_{G/G_{ad}}\Mod^{\C^\times})$ is a complex $M^*$ that is a module of $\CA_G$, together with an algebraic action of the group $G\ltimes \fg[1]$, such that $M^*$ is a weakly $G$-equivariant module of $\CA_G$, and the following commutation relation holds
    \be
\{b, c\}=\langle \xi_b^{ad}, c\rangle, \qquad \{\delta_{dR}, b\}= \mathrm{Lie}_b, \qquad b\in\fg[1],~ c\in \fg^*[-1].
    \ee
    Here $\xi_b^{ad}$ is the vector field on $G$ corresponding to $b$ (under the conjugation action), and $\langle -, -\rangle$ is the natural pairing between vector fields and 1-forms. The term $\mathrm{Lie}_b$ is the Lie algebra action of $\fg$ on $M^*$. In the next sections, we will construct the coproduct corresponding to the monoidal structure under the equivalence of equation \eqref{eq:DOequiv}, as well as a braiding. 
    
\end{Exp}

We will use the functoriality statement of Proposition \ref{Prop:DOKoszul} in what follows, so we give a quick discussion now. Let $f:\CX\to \CY$ be a quasi-compact morphism of smooth algebraic stacks, then it induces a map
\be
f^*: \Omega_{\CY}\to \Omega_{\CX}
\ee
of sheaf of algebras compatible with the de-Rham differential, namely a map $\CA_{\CY}\to \CA_{\CX}$. The functoriality of equation \eqref{eq:DOequiv} states that under the equivalence, the pushforward functor $f_*$ coincides with restriction functor. If $f$ is also smooth, then the pullback $f^*$ coincides with the induction functor. 

\begin{Exp}
    Let $f: X\to Y$ be a morphism of smooth affine varieties equivariant with respect to the action of $G$. Then it induces a map of algebras $\CA_Y\to \CA_X$ equivariant under the action of $T[1]G$. The push-forward functor is precisely the restriction functor along this map. On the other hand, if $\pi: X\to X/G$, then the functor $\pi^*$ is the forgetful functor for the action of $T[1]G$.

\end{Exp}

\subsubsection{Harish-Chandra modules and Chevalley-Eilenberg complex}\label{subsubsec:HCCE}

We now consider a similar Koszul duality statement for $\HC(\fg, K)$, and again in the equivariant context. Let $U_\hbar (\fg)$ be the Rees algebra associated to $U(\fg)$. It is a $\C^\times$-equivariant algebra where $\C^\times$ acts on $\fg$ and $\hbar$ with weight $-1$. A vector space $M$ is called an \textbf{asymptotic} Harish-Chandra module, if it is a module of $U_\hbar(\fg)$ and admits an action of $K$ such that $\hbar d\mathrm{act} (x)=\mathrm{Lie}_x$, where $d\mathrm{act}$ is the action of $\mathfrak{k}$ induced by the derivative of the $K$ action, and $\mathrm{Lie}_x$ is its action through the embedding $\mathfrak k\to \fg$. Since the relation $\hbar d\mathrm{act} (x)=\mathrm{Lie}_x$ is equivariant with respect to the $\C^\times$ action, one can define the category $\HC_\hbar (\fg, K)^{\C^\times}$ to be the category of $\C^\times$-equivariant asymptotic Harish-Chandra modules. 

On the other hand, consider the exterior algebra $\bigwedge \!\!{}^*\fg$, which admits a square-zero differential $d_{CE}$, the Chevalley-Eilenberg differential. Similar to the previous section, let us define $\Ce(\fg)$ to be the algebra
\be
\Ce(\fg):= \bigwedge \!\!{}^*\fg[\delta_{CE}]/(\delta_{CE}^2, \{\delta_{CE}, c\}=d_{CE}c).
\ee
This algebra admits an action of $K\ltimes \mathfrak k[1]$, where $K$ acts on $\Ce(\fg)$ via the adjoint action and $\mathfrak{k}[1]$ acts by contraction 
\be
\{b, c\}=\langle b, c\rangle, \qquad \forall b\in \mathfrak k[1], ~c\in \fg^*[-1].
\ee
An $\CA_{\fg, K}$ module is a module of $\Ce(\fg)$ together with a compatible action of $K\ltimes \mathfrak k[1]$, and we denote by $\CA_{\fg, K}\Mod^{\C^\times}$ the category of equivariant modules of $\CA_{\fg, K}$. In \cite[Section 7.7.7]{BDhecke}, Beilinson-Drinfeld defined a pair of adjoint functors
\be
\btik
C(\HC_\hbar (\fg, K)^{\C^\times}) \arrow[r, shift left, "\Omega_{\fg, K}"]& \arrow[l, shift left, "D_{\fg, K}"]C(\CA_{\fg, K}\Mod^{\C^\times})
\etik
\ee
Call a morphism of $\CA_{\fg, K}$ modules a $\fg$-quasiisomorphism if its image under $D_{\fg, K}$ is a quasi-isomorphism, and denote by $D_{qs}(\CA_{\fg, K}\Mod^{\C^\times})$ the localization of $K(\CA_{\fg, K}\Mod^{\C^\times})$ with respect to $\fg$-quasiisomorphisms. The following, as Proposition \ref{Prop:DOKoszul}, is a repeat of \cite[Section 7.7.7]{BDhecke} in the $\C^\times$-equivariant setting. 

\begin{Prop}\label{Prop:HCKoszul}

    There is an equivalence of triangulated categories
    \be
D(\HC_\hbar (\fg, K)^{\C^\times})\simeq D_{qs}(\CA_{\fg, K}\Mod^{\C^\times}).
    \ee
    
\end{Prop}

\begin{Rem}
    As before, a $\fg$-quasiisomorphism is an ordinary quasi-isomorphism, but the converse is not true. The RHS is again identified with the category of injective complexes or ind-coherent modules. 
\end{Rem}

The category $D(\HC_\hbar (\fg, K)^{\C^\times})$ does not admit a monoidal structure, but that of a module category of the Hecke category. The equivalence above gives a simple way to study the Hecke action. In the case when $G=K\times K$ and $K=K_\Delta$, the category $\HC_\hbar (\fg, K)^{\C^\times}=\HC_\hbar (K)^{\C^\times}$ does admit a monoidal structure. However, the above equivalence does not make it easier to study this monoidal structure. To simplify the situation, we consider the following specialization. Assume now that the embedding of Lie algebras $\mathfrak k\to \mathfrak g$ admits a splitting $\mathfrak g=\mathfrak k\oplus \mathfrak h$ such that $\fh$ is an ideal. The algebra $\Ce(\fh)$ admits an action of $K$ by the adjoint action. An $\CA_{\fh, K}$ module is a module of $\Ce(\fh)$ with a compatible action of $K$, namely, an object in $\Ce(\fh)\Mod (\Rep (K))$. 

We now build an adjoint pair
\be\label{eq:adjointHC}
\btik
C(\HC_\hbar (\fg, K)^{\C^\times}) \arrow[r, shift left, "\Omega_{\fh, K}"]& \arrow[l, shift left, "D_{\fh, K}"]C(\CA_{\fh, K}\Mod^{\C^\times})
\etik
\ee
similar to $\CA_{\fg, K}\Mod$. Consider the complex
\be
DR_{\fh, K}:= U_\hbar (\fg)\otimes \Ce(\fh), \qquad d=\sum_j (x_j)_R\otimes (c^j)_L-(\hbar)_R\otimes (\delta_{CE})_L.
\ee
Here $\{x_j\}\subseteq \fh$ is a set of basis of $\fh$ with dual basis $\{c^j\}$, and $(x_j)_R$ means right multiplication. This complex admits an action of $\Ce(\fh)$ on the right, as well as the structure of an asymptotic $(\fg\times \fk, K)$ module, and the two structures are compatible. Here $\fg$ acts by left multiplication on $U_\hbar(\fg)$, $\fk$ acts by right multiplication on $U_\hbar (\fg)$ as well as by conjugation on $\Ce(\fh)$ (multiplied by $\hbar\otimes 1$), and $K$ embeds into $G\times K$ diagonally. This object induces the pair of adjunctions above. More specifically, given an object $V\in C(\HC_\hbar (\fg, K)^{\C^\times})$, the functor $\Omega_{\fh, K}$ sends $V$ to
\be
\Hom_{U_\hbar (\fg)}(DR_{\fh, K}, V)= \Ce(\fh)^*\otimes V,
\ee
with the obvious $\CA_{\fh, K}$ module structure. On the other hand, given $W\in C(\CA_{\fh, K}\Mod^{\C^\times})$, the functor $D_{\fh, K}$ sends $W$ to 
\be
DR_{\fh, K}\otimes_{\Ce(\fh), U_\hbar (\fk)} W =U_\hbar (\fh) \otimes W. 
\ee

Call a morphism in $C(\CA_{\fh, K}\Mod^{\C^\times})$ a $\fg$-quasiisomorphism if its image under $D_{\fh, K}$ is a quasi-isomorphism. Similar to Proposition \ref{Prop:HCKoszul}, the following is true.

\begin{Prop}\label{Prop:HCKKoszul}
    Let $D_{qs}(\CA_{\fh, K}\Mod^{\C^\times})$ be the localization of $K(\CA_{\fh, K}\Mod^{\C^\times})$ with respect to $\fg$-quasiisomorphisms. Then there is an equivalence of triangulated categories
    \be
D(\HC_\hbar (\fg, K)^{\C^\times})\simeq D_{qs}(\CA_{\fh, K}\Mod^{\C^\times}). 
    \ee
\end{Prop}

\begin{proof}

    This is a repeat of the argument of \cite[Lemma 7.2.4]{BDhecke}. Indeed, one shows that for any $M\in C(\HC_\hbar (\fg, K)^{\C^\times})$, the canonical adjunction $D\Omega M\to M$ is a quasi-isomorphism. This implies that $N\to \Omega D N$ is a $\fg$-quasiisomorphism, and the result follows. 
    
\end{proof}

\begin{Rem}

    The relation between Proposition \ref{Prop:HCKoszul} and \ref{Prop:HCKKoszul} is that the former is a Koszul duality for the Lie algebra $\fg$ whereas the latter is a Koszul duality for $\fh$. The reason we do so is that the latter admits a Hopf structure in the example below.  
    
\end{Rem}

We now turn to the main example of our consideration, namely $G=K^2$ and $K=K_\Delta$.

\begin{Exp}\label{Exp:HCK}

    Consider $G=K^2$ and $K=K_\Delta$. In this case, we denote the Harish-Chandra category by $\HC_\hbar (K)^{\C^\times}$. The embedding $\fk\to \fk^2$ admits a splitting by an ideal $\fh=\fk\oplus 0$. The resulting action of $\fk$ on $\Ce(\fh)=\Ce(\fk)$ is simply the conjugation action of $\fk$ on its Chevalley-Eilenberg complex. Therefore, a module of $\CA_{\fh, K}$ is simply a module of $K$ together with a compatible action of $\Ce (\fk)$. In other words, we have
    \be\label{eq:HCCEK}
\CA_{\fh, K}\Mod\simeq \Ce (\fk)\Mod (\mathrm{Rep}(K)).
    \ee
Here we view $\Ce(\fk)$ as an algebra object in $\mathrm{Rep}(K)$. One can intuitively think of an object in the above categories as a module of the ``semi-direct product" $\C \cdot K\ltimes \Ce(\fk)$. Here $\C\cdot K$ is the group algebra. However, one needs to replace $\C\cdot K$ with a topological version in order to account for the smooth topology of the group $K$. We will do so in the next section and obtain a Hopf algebra structure on this semi-direct product. 
    
\end{Exp}

\section{Graded Hopf algebras}\label{sec:Hopfalgebra}

In the previous section, we introduced the category of D-modules and Harish-Chandra modules associated to an algebraic group $G$. Moreover,  by the work of \cite{BDhecke}, there are algebraic counterparts to these geometric categories. The examples of interest to us are Examples \ref{exp:G}, \ref{exp:G/G} and \ref{Exp:HCK}. 

In this section, we show that the algebraic counterparts in these examples all admit Hopf structures. The Hopf algebra structure is very natural for the algebra $\CA_G$, as it is the de-Rham complex for a group. We will start by deriving the Hopf algebra structure in the example of Harish-Chandra bimodules, as it is the most non-intuitive. This Hopf structure will be explained in the context of 1-shifted Lie bialgebras of \cite{NPShifted}. 

Unless otherwise stated, for the rest of this section, $G$ denotes an affine algebraic group with Lie algebra $\fg$. In this section and the next, we will need to use dual and Drinfeld double of infinite-dimensional Hopf algebras. The details of their construction are spelled out in Appendix \ref{app:double}. 

\subsection{Hopf algebras and Harish-Chandra modules}\label{subsec:HopfHC}

\subsubsection{Hopf algebras from 1-shifted Lie bialgebras}

A 1-shifted metric Lie algebra is a graded Lie algebra $\fh$ with a non-degenerate symmetric invariant bilinear pairing $\kappa$ of degree $1$. From $\kappa$, we may define a new invariant bilinear form $\wt\kappa(x, y)=\kappa (x, y)(-1)^{|x|}$. This bilinear form is skew-symmetric. Choose a pair of Lagrangian Lie subalgebras $\fa_\pm$ of $\fh$, the pairing $\wt\kappa$ identifies $\fa_+$ with the 1-shifted dual of $\fa_-$. In particular, this pairing induces cobrackets $\delta_\pm$ on $\fa_\pm$ which are dual to the brackets of $\fa_\mp$ under $\wt\kappa$. Moreover, $\delta_\pm$ are maps of the form
\be
\delta_\pm: \fa_\pm\to \mathrm{Sym}^2(\fa_\pm). 
\ee
The following was shown in \cite[Proposition 2.3]{NPShifted}. 

\begin{Prop}
    The cobrackets $\delta_\pm$ satisfies
    \be\label{eq:1LBS}
\delta_\pm\otimes 1+1\otimes \delta_\pm (\delta_\pm)=0\in \mathrm{Sym}^3(\fa_\pm), \qquad \delta_\pm ([X, Y])=[\delta_\pm (X), \Delta (Y)]+(-1)^{|X|}[\Delta (X), \delta_\pm (Y)].
    \ee

\end{Prop}

In \cite{pimenov} and \cite{NPShifted}, a Lie algebra $\fa$ together with a cobracket $\delta$ satisfying equation \eqref{eq:1LBS} is called a 1-shifted Lie bialgebra, and the triple $(\fh,\fa_+, \fa_-)$ is called the Manin-trimple corresponding to $\fa_+$. It was shown in \cite[Theorem 2.5]{NPShifted} that a 1-shifted Lie bialgebra $\fa$ gives rise to a canonical Manin-triple $(D(\fa), \fa, \fa^*[-1])$, and vice versa. 

In \cite{NPShifted}, it was explained that from such a Manin triple, one can naturally associate a graded Hopf algebra, whose construction we now recall. Let $\{x_i\}$ be a set of basis for $\fa_+$ with dual basis $\{c^i\}$ under $\wt\kappa$, and let $\mathbf{r}=\sum x_i\otimes c^i\in \fa_+\otimes \fa_-$. It turns out that the cobracket $\delta_\fh:= (-\delta_{\fa_+}, \delta_{\fa_-})$ is identified with $[\mathbf{r}, \Delta]$, where $\Delta$ is the symmetric coproduct of $\fh$. Let $U(\fh)$ be the universal enveloping algebra of $\fh$, and let $\C[\delta]$ be a free graded algebra generated by $\delta$ in degree $1$. We denote by $U(\fh, \delta):=\C[\delta]\ltimes U(\fh)$ the semi-direct product, where 
\be
[\delta, X]=\frac{1}{2}[\nabla \mathbf r, X],\qquad X\in U(\fh), \nabla: \fh\otimes \fh\to U (\fh).
\ee
Namely $\delta$ acts as the degree 1 inner differential $\frac{1}{2}[\nabla \mathbf r, -]$. One of the main results of \cite{NPShifted} is the following theorem.

\begin{Thm}\label{Thm:shiftedbialg}

    The algebra $U(\fh, \delta)$ has the structure of a graded Hopf algebra, where the coproduct of $\fh$ is the symmetric coproduct $\Delta$, and the coproduct of $\delta$ is
    \be
\Delta (\delta)=\delta\otimes 1+1\otimes \delta+\mathbf r. 
    \ee
    Moreover, assuming that $\fh_2=0$, then this Hopf algebra $U(\fh, \delta)$ admits a natural Hopf quotient, by the ideal generated by $\delta^2$. 
    
\end{Thm}

\begin{Rem}
    Note that because $\mathbf{r}$ is not symmetric, this Hopf algebra is not cocommutative. The antipode is the standard antipode on $\fh$ and $S\delta=-\delta+\sum x_ic^i$. The inverse of this antipode is $S^{-1}\delta=-\delta+\sum c^ix_i$. 
\end{Rem}

Let us now apply this to $\fh=T^*[-1]\fg$, the 1-shifted tangent Lie algebra of $\fg$, with Lagrangians and $\fa_+=\fg, \fa_-=\fg^*[-1]$. Since $\fg$ is in degree $0$, $\fh_2=0$ and therefore we have a graded Hopf algebra $U(T^*[-1]\fg, \delta)/\delta^2$. We denote this algebra by $H_\fg$. As an algebra, it is generated by $\fg, \fg^*[-1]$ and $\delta$, whose non-trivial brackets are
\be
[x_i, c^j]=f_{ki}^j c^k, \qquad [\delta, c^i]=-\frac{1}{2}\sum f^i_{jk}c^jc^k, \qquad \delta^2=0.
\ee
The bracket of $\delta$ with $c^i$ precisely gives the action of the Chevalley-Eilenberg differential. Namely, $\Ce(\fg)$ defined in Section \ref{subsubsec:HCCE}, is a subalgebra of $H_\fg$. Moreover, a module of $H_\fg$ is simply a module of $\Ce(\fg)$ with a compatible Lie algebra $\fg$ action. Namely
\be
H_\fg\Mod\simeq \Ce(\fg)\Mod (\Rep (\fg)).
\ee
As in equation \eqref{eq:HCCEK}, we would like to replace $\Rep (\fg)$ above by $\Rep (G)$. To do so, we would like to replace the Lie algebra $\fg$ by the algebraic group $G$. To correctly account for the representation theory of the algebraic group, it is necessary to include a natural topology, as in \cite[Section 8.3.4]{DN24}. More precisely, let $\CO_G$ be the Hopf algebra of algebraic functions on $G$. An algebraic $G$ module is precisely a comodule for $\CO_G$. Let $\Dist (G)$ be the linear dual of $\CO_G$, which admits a pro-finite topology dual to the filtration of $\CO_G$ by finite-dimensional subspaces. Being dual to $\CO_G$, $\Dist (G)$ naturally admits a topological Hopf algebra structure, and contains both $U(\fg)$ and the group algebra $\C\cdot G$ as Hopf subalgebras via the following formula
\be\label{eq:DistGg}
\begin{aligned}
& g\in G\mapsto \delta_g (f)=f(g), ~~\forall f\in \CO_G\\
& X\in \fg\mapsto \pd_X (f)=\xi_X^Rf(e), ~~\forall f\in \CO_G
\end{aligned}
\ee
Here the formula $\xi_X^Rf(e)$ means the derivative of $f$ with respect to the left-invariant vector field $X$, evaluated at the identity element $e\in G$. 

Let us now define a topological Hopf algebra $H_G$. As a vector space, it is defined by $H_G=\Dist (G) \otimes \Ce(\fg)$. The algebra $\Ce(\fg)$ is an algebra object in the category $\Rep (G)$, which is identified with comodules of $\CO_G$. Let $\rho: \Ce (\fg)\to \CO_G\otimes \Ce (\fg)$ be the induced coaction, which is a map of algebras. We define an algebra structure on $H_G$ by
\be
(\varphi_1 \otimes \omega_1)* (\varphi_2\otimes \omega_2)= (\varphi_1\varphi_2^{(1)}\otimes \rho (\omega_1)(S\varphi_2^{(2)})\cdot \omega_2),\qquad \varphi_i\in \Dist (G), \omega_i\in \Ce (\fg). 
\ee
Here $\Delta(\varphi_2)=\varphi_2^{(1)}\otimes \varphi_2^{(2)}$ is the Sweedler notation and $\rho (\omega_1)(S\varphi_2^{(2)})$ is the element in $\Ce (\fg)$ given by pairing the first factor of $\rho (\omega_1)$ with $S\varphi_2^{(2)}$. We suggestively write this as $S\varphi_2^{(2)}\rhd \omega_1$, signifying that its the group action on $\Ce(\fg)$.  It is easy to see that this formula is smooth with respect to the topology of $H_G$.

\begin{Lem}
    The above formula makes $H_G$ into a topological algebra. 
\end{Lem}

\begin{proof}
    We just need to prove associativity. Let $\varphi_i \otimes \omega_i\in H_G$ be given, for $1\leq i\leq 3$, we have
    \be
    \begin{aligned}
\lp  (\varphi_1 \otimes \omega_1)* (\varphi_2\otimes \omega_2)\rp * (\varphi_3\otimes\omega_3)&=(\varphi_1\varphi_2^{(1)}\otimes (S\varphi_2^{(2)}\rhd \omega_1)\cdot \omega_2) * (\varphi_3\otimes \omega_3)  \\ &=(\varphi_1\varphi_2^{(1)}\varphi_3^{(1)}\otimes (S\varphi_3^{(2)}\rhd \wt\omega_1)\cdot \omega_3)
    \end{aligned}
    \ee
    where $\wt\omega_1=(S\varphi_2^{(2)}\rhd \omega_1)\cdot \omega_2$. Since $\rho$ is a map of algebras, we have
    \be
S\varphi_3^{(2)}\rhd \wt\omega_1=\rho (\rho(\omega_1)(S\varphi_2^{(2)}))(S\varphi_3^{(21)})\rho (\omega_2)(S\varphi_3^{(22)})=(S(\varphi_2^{(2)}\varphi_3^{(21)})\rhd\omega_1)(S\varphi_3^{(22)}\rhd \omega_2). 
    \ee
    Therefore we have
    \be
\lp  (\varphi_1 \otimes \omega_1)* (\varphi_2\otimes \omega_2)\rp * (\varphi_3\otimes \omega_3)= \varphi_1\varphi_2^{(1)}\varphi_3^{(1)}\otimes (S(\varphi_2^{(2)}\varphi_3^{(21)})\rhd\omega_1)(S\varphi_3^{(22)}\rhd \omega_2)\omega_3. 
    \ee
    On the other hand, we have
    \be
\begin{aligned}
    (\varphi_1 \otimes \omega_1)* \lp  (\varphi_2 \otimes \omega_2)* (\varphi_3\otimes \omega_3)\rp&= (\varphi_1\otimes \omega_1)* (\varphi_2\varphi_3^{(1)},  (S\varphi_3^{(2)}\rhd \omega_2)\cdot \omega_3)\\ &= \varphi_1\varphi_2^{(1)}\varphi_3^{(11)}\otimes  (S(\varphi_2^{(2)}\varphi_3^{(12)})\rhd \omega_1)(S\varphi_3^{(2)}\rhd \omega_2)\omega_3.
\end{aligned}
    \ee
    These two expressions coincide. 
\end{proof}

\begin{Rem}
    A simpler proof would be to recognize that the map $\C\cdot G\to \Dist (G)$ is dense, and therefore we only need to check the relations on $\C\cdot G$, which is clear. The algebra $H_G$ is the semi-direct product of $\Dist (G)$ with $\Ce (\fg)$, and we will use the notation $\Dist (G)\ltimes \Ce (\fg)$ in what follows. 
\end{Rem}

The relation between $H_G$ and $H_\fg$ is clear. The embedding $U(\fg)\to \Dist (G)$, together with the unit map of $\Ce (\fg)$, extends to a map of algebras from $H_\fg$ to $H_G$. The following is a simple proposition, whose proof we omit. 

\begin{Prop}
    The coproduct on $\Ce (\fg)$ defined in $H_\fg$ makes $H_G$ into a topological Hopf algebra, and the embedding $H_\fg\to H_G$ an embedding of Hopf algebras. 
    
\end{Prop}

\subsubsection{An equivalence of monoidal categories}

It is clear from the definition of $H_G$ that a smooth module of $H_G$ is simply an algebraic representation of $G$, equipped with a compatible action of $\Ce(\fg)$. Namely, we have an equivalence of abelian categories:
\be
H_G\Mod\simeq \Ce(\fg)\Mod (\mathrm{Rep}(G))\simeq \CA_{\fg, G}\Mod. 
\ee
Note that the $\C^\times$ action on $\Ce(\fg)$ extends to an action on $H_G$. Proposition \ref{Prop:HCKKoszul} can be re-stated in terms of $H_G$ as an equivalence of triangulated categories
\be
D(\HC_\hbar (G)^{\C^\times})\simeq D_{qs} (H_G\Mod^{\C^\times}).
\ee
We now show the following statement. 

\begin{Thm}\label{Thm:HCmonoidal}
 There is an equivalence of triangulated monoidal categories
 \be
D(\HC_\hbar (G)^{\C^\times})\simeq D_{qs} (H_G\Mod^{\C^\times}).
 \ee   
\end{Thm}
 
\begin{proof}
    The equivalence here is induced by the adjoint pair $(\Omega_{\fg, G}, D_{\fg, G})$ in equation \eqref{eq:adjointHC}. We will show that the functor
    \be
D_{\fg, G}: C(H_G\Mod^{\C^\times})\longrightarrow C(\HC_\hbar (G)^{\C^\times})
    \ee
    can be given a monoidal structure.

    To do so, we need the explicit chain-level model of $D_{\fg, G}$. Let $V=\bigoplus_{i, j} V_{j}^i$ be a complex of $H_G$ modules, where the supscript $i$ denotes homological grading and subscript $j$ denotes the $\C^\times$ weight.  The functor $D_{\fg, G}$ applied to $V$ is given by
\be
D_{\fg, G} (V)=U_\hbar (\fg)\otimes V, \qquad d=d_V+\sum (x_a)_R\otimes \epsilon^a-\hbar\otimes \delta
\ee
Since $D_{\fg, G} (V)$ is an object in $C(\HC_\hbar(G)^{\C^\times})$, it is naturally a bimodule over $U_\hbar (\fg)$, whose right module structure is given by $(x_a)_L-\hbar d\mathrm{Act}_{D_{\fg, G} (V)} (x_a)$, where $d\mathrm{Act}_{D_{\fg, G} (V)} (x_a)$ is the derivative of the action of $G$ on $D_{\fg, G} (V)$. Let $V, W$ be two objects in $C (H_G\Mod^{\C^\times})$, there is a natural identification of the product $D_{\fg, G} (V)\otimes_{\HC}D_{\fg, G} (W)$ with
\be
D_{\fg, G} (V)\otimes_{\HC}D_{\fg, G} (W))=\lp U_\hbar (\fg)\otimes V\rp \otimes_{U_\hbar (\fg)} \lp U_\hbar (\fg)\otimes W \rp\simeq U_\hbar (\fg)\otimes V\otimes W,
\ee
whose $U_\hbar (\fg)$ module structure is given by its left multiplication, whose $G$-action is the obvious diagonal action, and whose differential is given by
\be
d_V+d_W+ \sum (x_a)_R\otimes \epsilon^a_V+\sum \lp (x_a)_L+ \hbar d\mathrm{Act}_{D_{\fg, G} (V)} (x_a) \rp \otimes \epsilon^a_W-\hbar\otimes (\delta_V+\delta_W).
\ee
Here the term $(x_a)_L- \hbar d\mathrm{Act}_{D_{\fg, G} (V)} (x_a)$ comes from the right action of $U_\hbar (\fg)$ on $D_{\fg, G} (V)$. We denote by $J_{V, W}$ the above natural identification. Now if we use the fact that $(x_a)_L- \hbar d\mathrm{Act}_{D_{\fg, G} (V)} (x_a)= (x_a)_R-\hbar d\mathrm{Act}_V(x_a)$, we find that the above differential becomes
\be
J_{V, W}(d)= d_V+d_W+ \sum (x_a)_R\otimes (\epsilon^a_V+\epsilon^a_W)- \hbar\otimes (\delta_V+\delta_W+d\mathrm{Act}_V(x_a)\otimes \epsilon^a_W).  
\ee
Comparing this with the coproduct of Theorem \ref{Thm:shiftedbialg}, we find that $J_{V, W}(d)$ is precisely the differential on $D_{\fg, G} (V\otimes_H W)$, namely
\be
J_{V, W}(d)=d_V+d_W+\sum (x_a)_R\otimes \Delta (\epsilon^a)-\hbar \otimes \Delta (\delta). 
\ee
Consequently, $J_{V, W}$ gives a natural identification
\be
J_{V, W}: D_{\fg, G} (V\otimes_H W)\to D_{\fg, G}(V)\otimes_{\HC} D_{\fg, G} (W).
\ee
To show that $(D_{\fg, G}, J)$ is a monoidal functor, we need to show that $J$ preserves associativity, which is clear from the construction. The descent of this monoidal functor to the category $D_{qs}(H_G\Mod^{\C^\times})$ gives the monoidal equivalence.

\end{proof}

\subsection{Hopf algebras from de-Rham complexes}

\subsubsection{de-Rham complex of a group}

We consider Example \ref{exp:G}. The multiplication map $m_G: G\times G\to G$ induces a co-product on the de-Rham complex of $G$, and similarly on $\CA_G$. We now identify this coproduct. We can identify $\Omega_G$ with functions on $T[1]G=G\ltimes \fg[1]$, where we identify $\fg[1]$ with right-invariant vector fields. 

\begin{Lem}\label{Lem:Ogroup}
    Under the identification $T[1]G=G\ltimes \fg[1]$, the group structure is identified with the extension of the group $G$ by its adjoint module $\fg[1]$. 
\end{Lem}

\begin{proof}
Given two element $(g_i, v_i)\in T[1]G$, the claim is that their multiplicatoin is given by
\be
(g_1, v_1)\cdot (g_2, v_2)= (g_1g_2, (g_2^{-1}v_1g_2)+v_2).
\ee
To prove this, let $f\in \CO_G$ be a function, we need to compute the value of $(m_G)_*(g_i, v_i)(f)$. By definition, this is equal to
\be
(v_1, v_2)(m_G^*(f))(g_1,g_2)=\pd_t\vert_{t=0} m_G^*(f)(g_1e^{tv_1}, g_2e^{tv_2})=\pd_t\vert_{t=0} f(g_1g_2 e^{tg^{-1}_2v_1g_2}e^{tv_2})
\ee
and the result follows. 
    
\end{proof}

The following corollary is now a consequence of Lemma \ref{Lem:Ogroup} and the functoriality of $D_\hbar$ in Proposition \ref{Prop:DOKoszul}.

\begin{Cor}\label{Cor:coprodCAG}

    The algebra $\CA_G$ is a Hopf algebra whose coproduct is given by
    \be\label{eq:coprodTG}
\Delta (\delta_{dR})=\delta_{dR}\otimes 1+1\otimes \delta_{dR}, \qquad \Delta (f)=m_G^*(f), \qquad \Delta (c)=1\otimes c+\rho (c)^{op},
    \ee
    for all $f\in \CO_G, c\in \fg^*[-1]$. Here $\rho: \fg^*[-1]\to \CO_G\otimes \fg^*[-1]$ is the co-action of $\CO_G$ on $\fg^*[-1]$. Moreover, there is an equivalence of triangulated monoidal categories
    \be
\CD_\hbar (G)^{\C^\times}\simeq D_{qs}(\CA_G\Mod^{\C^\times}). 
    \ee
    
\end{Cor}

\begin{proof}
    First, \eqref{eq:coprodTG} indeed defines a Hopf algebra structure on $\CA_G$. The fact its restriction on $\Omega_G$ is a Hopf algebra structure follows from the fact that $\Omega_G$ is identified with functions on the (DG) group $T[1]G$, and equation \eqref{eq:coprodTG} is simply the corresponding map on the algebra of functions. The fact that it extends to the generator $\delta_{dR}$ follows from the fact that for any map $h: X\to Y$, one has $h^*(d_{dR}\omega)=d_{dR}h^*(\omega)$ for any $\omega\in \Omega(X)$; therefore $m_G^*([\delta_{dR}, -])=[\delta_{dR}\otimes 1+1\otimes \delta_{dR}, m_G^*(-)]$.

    We are left to show that the equivalence of Proposition \ref{Prop:DOKoszul} is one of monoidal categories. To this end, we have the following diagram of categories
    \be
\btik
 C(\CA_G\Mod^{\C^\times})\times C(\CA_G\Mod^{\C^\times})\rar{\varphi}\dar{D_\hbar\times D_\hbar} & C(\CA_{G\times G}\Mod^{\C^\times})\rar{\psi}\dar{D_\hbar} & C(\CA_{G}\Mod^{\C^\times})\dar{D_\hbar}\\
 C(\CM_\hbar (G)^{\C^\times})\times C(\CM_\hbar (G)^{\C^\times})\rar{\hbar_1=\hbar_2} & C(\CM_\hbar (G\times G)^{\C^\times})\rar{m_{G, *}} & C(\CM_\hbar (G)^{\C^\times})
\etik
    \ee
The horizontal diagrams define the monoidal product on the categories involved. The functor $\varphi$ is pullback from $\CA_G\otimes \CA_G$ to $\CA_{G\times G}$ which is identity on $\Omega_{G\times G}$ and maps $\delta_{dR, G\times G}$ to $\delta_{dR}\otimes 1+1\otimes \delta_{dR}$, and the functor $\psi$ is the pullback functor along the coproduct of $\Omega_G$. On the second diagram, one first take the tensor product of modules over $\C\lbb\hbar\rbb$, then pushforward along $m_G$. By Proposition \ref{Prop:DOKoszul}, the second square commutes, we just need to verify that the first diagram commutes. 

Recall from equation \eqref{eq:DRhbar}, the functor $D_\hbar$ sends $N\in C(\CA_G\Mod^{\C^\times})$ to $D_G^\hbar \otimes_{\CO_X} N$ with Koszul differential. Therefore, we have
\be
D_\hbar (N_1)\otimes_{\C\lbb\hbar\rbb} D_\hbar (N_2)=(D_{G}^\hbar\otimes_{\C\lbb\hbar\rbb} D_G^\hbar) \otimes_{\CO_{G\times G}} (N_1\otimes N_2)=D_{G\times G}^\hbar \otimes_{\CO_{G\times G}} (N_1\otimes N_2),
\ee
with differential
\be
\sum v_i^{(1)}\otimes \omega^{i, (1)}+v_i^{(2)}\otimes \omega^{i, (2)}+\hbar \otimes (\delta_{dR}^{(1)}+ \delta_{dR}^{(2)})+d_{N_1}+d_{N_2},
\ee
where the supscripts $(1), (2)$ refer to the action on the first and second factor of $D_G^{\hbar}$ (and $N_1\otimes N_2$). The above DG module coincides with $D_\hbar \varphi(N_1\otimes N_2)$.

The commutativity of the diagram gives the monoidal isomorphism for the functor $D_\hbar$, and the fact that it is associative is clear. The descent of this monoidal functor to $D_{qs}$ gives the desired equivalence at the level of monoidal categories.

\end{proof}

\subsubsection{de-Rham complex of the character stack}

We now consider Example \ref{exp:G/G}. Recall that the category $\CA_{G/G_{ad}}\Mod^{\C^\times}$ is the category of modules of $\CA_G$ with a compatible action of $T[1]G$. To organize this into an algebra and correctly account for the topology of $G$, we consider a topological algebra similar to Section \ref{subsec:HopfHC}. Let us denote by $A_{G/G_{ad}}$ the topological vector space
\be
A_{G/G_{ad}}=\Dist (G)\otimes \bigwedge \!{}^*\fg[1]\otimes \CA_G.
\ee
Define an algebra structure on this space via the commutation relation
\be
 \omega\varphi=\varphi^{(1)}(S\varphi^{(2)}\rhd \omega), \qquad [b, c]=\langle \xi_b^{ad}, c\rangle, \qquad [\delta_{dR}, b]=\mathrm{Lie}_b,
\ee
together with the commutation relations of $\CA_G$. Here $\omega\in \bigwedge \!{}^*\fg[1]$ or $\bigwedge \!{}^*\fg^*[-1]$, $b\in \fg[1], c\in \fg^*[-1]$ and $\mathrm{Lie}_b$ is the action of the Lie algebra element corresponding to $b$ associated with the action of $K$. There is an action of $\C^\times$ under which the weights and grading coincides. It is now clear from the discussion of Example \ref{exp:G/G} that a module of $A_{G/G_{ad}}$ is precisely an object in $\CA_{G/G_{ad}}\Mod^{\C^\times}$. Therefore, we have an equivalence of triangulated categories
\be\label{eq:DOtriequiv}
\CD_\hbar (G/G_{ad})^{\C^\times}\simeq D_{qs}(A_{G/G_{ad}}\Mod^{\C^\times}).
\ee
We now upgrade this to a monoidal equivalence.

\begin{Thm}\label{Thm:cDOmonoidal}
    There is a Hopf algebra structure on $A_{G/G_{ad}}$ extending the Hopf structure of $\CA_G$. Moreover, there is an equivalence of triangulated monoidal categories
    \be
\CD_\hbar (G/G_{ad})^{\C^\times}\simeq D_{qs}(A_{G/G_{ad}}\Mod^{\C^\times}).
    \ee
\end{Thm}

We devide the proof into steps. 

\begin{Lem}\label{Lem:AGGcom}
    The natural Hopf structure of $\CA_G$ and $\Dist (G)$, together with $\Delta (b)=b\otimes 1+1\otimes b$ and $\epsilon(b)=0$, makes $A_{G/G_{ad}}$ into a topological Hopf algebra. 
\end{Lem}

\begin{proof}
  We show that $\Delta$ does extend to a map of algebras by showing that it preserves the commutation relation defining $A_{G/G_{ad}}$. It clearly preserves the commutation relation of $\Dist (G)$ with $b, c$ as well as that of $\delta_{dR}$ with $b, c$, using the fact that $\Delta (\mathrm{Lie}_b)=\mathrm{Lie}_b\otimes 1+1\otimes \mathrm{Lie}_b$ in $\Dist (G)$. We therefore only need to show that $\Delta [b, c]=[\Delta(b), \Delta (c)]$. Both are functions on $G\times G$, so let $g_1, g_2\in G$, then
    \be
\Delta[b,c](g_1, g_2)=\langle \xi_b^{ad}, c\rangle (g_1g_2).
    \ee
    We need to calculate this more explicitly. The element $\xi_b^{ad}$ is the adjoint vector field associated to $b$, which is equal to $-\xi_b^L+\xi_b^R$ where $\xi_b^L$ is the right invariant vector field and $\xi_b^R$ is the left-invariant vector field. Since $\xi_b^R$ form a basis for the global sections of vector fields on $G$, we can write
    \be
\xi_b^L=S\otimes 1\rho^\vee (\xi_b^R), \qquad \rho^\vee: \fg\to \CO_G\otimes \fg. 
    \ee
    Therefore, the above evaluates to 
    \be
\langle \xi_b^{ad}, c\rangle (g_1g_2)=-\langle \xi_b^L, c\rangle (g_1g_2)+\langle \xi_b^R, c\rangle (g_1g_2)=\epsilon (g_1g_2) (b, c)-(g_2^{-1}g_1^{-1}b, c)=((1-g_2^{-1}g_1^{-1})b, c). 
\ee
Here $(b,c)$ is the natural pairing between $\fg$ and $\fg^*$. On the other hand, we have
\be
[\Delta(b), \Delta (c)](g_1, g_2)=[1\otimes b, 1\otimes c](g_1, g_2)+[b\otimes 1,\rho(c)^{op}](g_1, g_2). 
\ee
The first term above is equal to $((1-g_2^{-1})b, c)$, while the second term is $((1-g_1^{-1})b, g_2c)=((g_2^{-1}-g_2^{-1}g_1^{-1}) b, c)$. Combining both, we find
\be
[\Delta(b), \Delta (c)](g_1, g_2)= ((1-g_2^{-1}g_1^{-1}) b, c)=\Delta[b, c](g_1, g_2),
\ee
as desired. The action of the antipode $S$ can be read-off from the above coproducts easily. 

\end{proof}

Let us now consider the monoidal structure of $\CD_\hbar (G/G_{ad})^{\C^\times}$. This is given by identifying $G/G_{ad}$ with $G_{\Delta}\setminus G\times G/G_{\Delta}$ and use the convolution monoidal structure of the Hecke category. We can also identify $G\times G$ with $G_\Delta\ltimes G_L$ where $G_L$ is $G\times 1\subseteq G\times G$, and the action of $G_\Delta$ is by conjugation. We now consider the following diagram of correspondence
\be
\btik
(G_L/G_\Delta)^2 &\arrow[l, swap, "p"] \rar{m} G_L^2/G_\Delta & G_L/G_\Delta
\etik
\ee
Here $p$ corresponds to the diagonal embedding $G_\Delta\to G_\Delta^2$ and $m$ is the multiplication on $G_L^2$. The functor $m_*p^*$ defines a functor
\be
\btik
\lp \CD_\hbar (G_L/G_\Delta)^{\C^\times}\rp^2\rar{m_*p^!} & \CD_\hbar (G_L/G_\Delta)^{\C^\times}.
\etik
\ee
This, together with the obvious associativity morphism, defines a monoidal structure on $G_L/G_\Delta$. We prove the following statement, which also spells out the associativity isomorphism explicitly. 

\begin{Lem}\label{Lem:G/Gassoc}

    The embedding $G_L\to G^2$ identifies the above monoidal structure with the convolution monoidal structure. 
    
\end{Lem}

\begin{proof}
    This is a tedious diagram chasing. First, we have the following commutative diagrams
    \be\label{eq:monoidalGG}
\btik
(G_L/G_\Delta)^2 \dar &\arrow[l, swap, "p"] \rar{m} G_L^2/G_\Delta \dar& G_L/G_\Delta\dar\\
(G_\Delta\setminus G^2/G_\Delta)^2 &\rar{\wt m}\arrow[l, swap, "\wt p"] G_\Delta\setminus G^2\times_{G_\Delta} G^2/G_\Delta  & G_\Delta\setminus G^2/G_\Delta
\etik
    \ee
    where the vertical arrows are isomorphisms. This gives the isomorphism of the monoidal products. We need to show that this isomorphism preserves the associativity morphisms. To understand the associativity axiom, we consider the following base-change diagram, which computes $(M\otimes_{G/G_{ad}}N)\otimes_{G/G_{ad}} P$. 
    \be
\btik
 & & \wt C \arrow[dl, swap, "\wt p"]\arrow[dr, "\wt m_{12}"] & & \\
 & C\times X\arrow[dl, swap, "p_{12}"]\arrow[dr, "m_{12}"] & & C \arrow[dl, swap, "p"]\arrow[dr, "m"]  & \\
X^3 & & X^2 & & X
\etik
    \ee
    Heree $X= G_\Delta\setminus G^2/G_\Delta$, $C= G_\Delta\setminus (G^2)^{\times_{G_\Delta} 2}/G_\Delta$, and $\wt C=G_\Delta\setminus (G^2)^{\times_{G_\Delta} 3}/G_\Delta$. Base-change property of smooth morphisms says that there is a natural isomorphism $p^!m_{12, *}=\wt m_{12, *}\wt p^!$, which supplies a natural isomorphism
\be
m_*p^!m_{12, *}p_{12}^!\longrightarrow m_{123, *}p_{123}^!,\qquad p_{123}=p_{12}\wt{p},~ m_{123}=m \wt m_{12}.  
\ee
A similar diagram that calculates $M\otimes_{G/G_{ad}}(N\otimes_{G/G_{ad}} P)$ gives rise to a similar natural isomorphism, 
\be
m_*p^!m_{23, *}p_{23}^!\longrightarrow m_{123, *}p_{123}^!,  
\ee    
and the associativity isomorphism is the composition of these two natural isomorphisms. 

Note that a similar diagram for $G_L/G_\Delta$ exists as well. At this point all we need to show is that the embedding $G_L\to G^2$ identifies the base-change diagram
\be
\btik
G_\Delta\setminus G^2\times_{G_\Delta}G^2\times_{G_\Delta}G^2/G_\Delta \rar{\wt m_{12}}\dar{\wt p}&G_\Delta\setminus G^2\times_{G_\Delta}G^2/G_\Delta \dar{p}\\
G_\Delta\setminus G^2\times_{G_\Delta}G^2/G_\Delta\times  G_\Delta\setminus G^2/G_\Delta \rar{m_{12}} &  (G_\Delta\setminus G^2/G_\Delta)^2
\etik
\ee
with the base-change diagram
\be
\btik
G_L^3/G_\Delta \rar{\wt{m}_{12}}\dar{\wt p}& G_L^2/G_\Delta\dar{p}\\
 G_L^2/G_\Delta\times G_L/G_\Delta \rar{m_{12}}& (G_L/G_\Delta)^2  
\etik
\ee
which is clear. This shows that the identification preserves the associativity isomorphism, as desired.

\end{proof}

Now we proceed to the proof of Theorem \ref{Thm:cDOmonoidal}.

\begin{proof}[Proof of Theorem \ref{Thm:cDOmonoidal}]

    We have defined the algebra $A_{G/G_{ad}}$. One can similarly define algebra $A_{G^n/G_{ad}}$, which are simply the semi-direct product of $\Dist (G)\otimes \bigwedge {}^*\fg[1]$ with $\CA_G^{\otimes n}$ (except there is only one generator $\delta_{dR}$, whose coproduct can be similarly treated as in Corollary \ref{Cor:coprodCAG}). One similarly has an equivalence of triangulated categories
\be
\CD_\hbar (G^n/G_{ad})^{\C^\times}\simeq D_{qs}(A_{G^n/G_{ad}}\Mod^{\C^\times}).
\ee
Note that for an object $M$ in the LHS, the vector space underlying the corresponding object $\Omega(M)$ on the RHS is simply the global section of the de-Rham complex of $M$ on $G^n$. 

Let us consider the top row in the diagram of equation \eqref{eq:monoidalGG}. Under the above equivalence, pull-back along the morphism $p$  corresponds to restriction functor along an algebra morphism
\be
\Omega (p): A_{G^2/G_{ad}}\longrightarrow A_{G/G_{ad}}\wh\otimes  A_{G/G_{ad}}, 
\ee
whose restriction on $\CA_G^{\otimes 2}$ is identity and whose restriction on $\Dist (G)\ltimes \bigwedge \!{}^*\fg[1]$ is the symmetric coproduct. On the other hand, pushforward along $m$ corresponds to restriction along an algebra morphism
\be
\Omega (m):  A_{G/G_{ad}}\longrightarrow  A_{G^2/G_{ad}},
\ee
which is identity on $\Dist (G)\ltimes \bigwedge \!{}^*\fg[1]$ and whose restriction to $\CA_G$ is simply the coproduct of $\CA_G$. Consequently, the composition $\Omega (m)\circ \Omega (p)$ is identical to the coproduct of $A_{G/G_{ad}}$. 

To finish the proof, we need to show that the associativity isomorphisms match under this isomorphism, namely that it is identity. As we have seen from Lemma \ref{Lem:G/Gassoc}, the complication of the associativity isomorphism is encoded in the base-change isomorphism of the diagram
\be
\btik
G_L^3/G_\Delta \rar{\wt{m}_{12}}\dar{\wt p}& G_L^2/G_\Delta\dar{p}\\
 G_L^2/G_\Delta\times G_L/G_\Delta \rar{m_{12}}& (G_L/G_\Delta)^2  
\etik
\ee
If we view this base-change diagram as a diagram of $G_\Delta^2$-equivariant schemes as follows
\be
\btik
G_\Delta\times G_L^3 \rar{\wt{m}_{12}}\dar{\wt p}& G_\Delta\times G_L^2\dar{p}\\
 G_L^2\times G_L \rar{m_{12}}& (G_L)^2  
\etik
\ee
then the base-change isomorphism on the category of $\CA$-modules is identity, since all the maps involved are morphisms of smooth affine varieties and $p$ is a smooth fibration. Now the result follows by applying the canonical equivalence
\be
A_{G_\Delta\times G_L^2/G_\Delta^2}\Mod^{\C^\times}\simeq A_{G_L^2/G_\Delta}\Mod^{\C^\times},
\ee
which is given by taking $T[1]G_\Delta$-invariants for one of the two copies of $G_\Delta$. This completes the proof of the theorem.

\end{proof} 

\subsubsection{Loop space and its double}

 We have seen that $T[1]G=G\ltimes \fg[1]$ has the structure of an algebraic group. Functions on $T[1]G$ can be identified with the Hopf subalgebra $\Omega_G$ of $\CA_G$. Let $\Dist (T[1]G)$ be the linear dual of $\Omega_G$, which admits a topological Hopf algebra structure, whose topology is similar to the topology of $\Dist (G)$. The topological vector space
\be
D(\Omega_G):= \Omega_G \otimes \Dist (T[1]G)
\ee
admits the structure of a topological quasi-triangular Hopf algebra, where the $R$-matrix is the standard $R$-matrix of the double of $\Omega_G$. Namely, if one choose a set of basis $\{f_i\}\in \Omega_G$ with dual basis $\{f^i\}\in \Dist (T[1]G)$, then $\CR_\Omega=\sum (-1)^{|f_i|} f_i\otimes f^i$. 

Let us identify $\Omega_G$ with $\CO_G\otimes \!\bigwedge {}^*\fg^*[-1]$, so that its linear dual is identified with $\Dist (G)\otimes \bigwedge \!{}^*\fg[1]$. The identification is via
\be\label{eq:hopfpairingO}
(f\cdot \omega^I, \varphi\cdot \omega_I)=(f,\varphi)(\omega^I, \omega_I).
\ee
As we shown in \cite[Section 8.4]{DN24}, this is a Hopf pairing and makes $D(\Omega_G)$ into a semi-direct product $\Dist (T[1])G\ltimes \Omega_G$. 

\begin{Lem}\label{Lem:DAemb}

There is an embedding $D(\Omega_G)\to A_{G/G_{ad}}$ of Hopf algebras.

\end{Lem}

\begin{proof}

It is clear that the commutation relations of $\Dist (G)$ with $\Omega_G$ are identical in $D(\Omega_G)$ and $A_{G/G_{ad}}$. Let $b\in \fg[1]$ and $c\in \fg^*[-1]$. The coproduct for $b$ is the standard symmetric coproduct, whereas the coproduct for $c$ is given by
\be
\Delta\otimes 1\Delta (c)= 1\otimes 1\otimes c+1\otimes \rho(c)^{op}+1\otimes \Delta (\rho(c)^{op}).
\ee 
Using the formula in \eqref{eq:doubleprod} and the fact that $S^{-1}b=Sb=-b$, we find
\be
b\cdot c=-c\cdot b-(c, b)+((-)^{-1}b, c). 
\ee
Moreover, the element $f$ and $b$ commutes with each other since their components have trivial pairing. Therefore, sending $b\to -b$ gives an embedding of algebras. It is clearly an embedding of Hopf algebras.

\end{proof}

In the next section, we show that the $R$-matrix of $D(\Omega_G)$ gives rise to an $R$-matrix for $A_{G/G_{ad}}$. Although this can be done by direct calculation, we prefer to do go through a different route, by computing the double of $\CA_G$ and $H_G$.

\section{Quasi-triangular structure and Hopf twist}\label{sec:double}

In this section, we show that the algebra $A_{G/G_{ad}}$ is a quasi-triangular Hopf algebra. The idea is to show that it can be identified with a suitable quotient of the double of $\CA_G$. The following is the main statement of this section.

\begin{Thm}\label{Thm:double}

The following two statements hold. 

\begin{enumerate}

\item The double $D(\CA_G)$ is a central extension of $A_{G/G_{ad}}$ by a primitive element $\delta_{dR}^*$ (a suitable dual of $\delta_{dR}$). Therefore $D(\CA_G)/\delta^*=A_{G/G_{ad}}$ is a quasi-triangular Hopf algebra.

\item There is a twist equivalence between $D(\CA_G)$ and $D(H_G)$. 

\end{enumerate}

\end{Thm}

\begin{Rem}
These results are expected as in Theorem \ref{Thm:center=character}. However, we note that Theorem \ref{Thm:double} only implies an abelian equivalence between the centers of the two categories, not the triangulated categories. It is very believable that this equivalence still holds in the triangulated case. We also note that $D_{qs}(D(\CA_G)\Mod^{\C^\times})$ should be the center of the category $\CD_\hbar (G)^{\C^\times}$, which was mis-identified with $D_{qs}(A_{G/G_{ad}}\Mod^{\C^\times})$ in \cite{BGngo}. 

\end{Rem}

\subsection{The double of $\CA_G$}

Let us compute the double of $\CA_G$. We identify $\CA_G$ with the vector space $\CO_G\otimes \bigwedge\!{}^*\fg^*[-1]\otimes \C[\delta_{dR}]$ via multiplication. We can extend the pairing of $\Omega_G$ and $\Dist (T[1]G)$ in equation \eqref{eq:hopfpairingO} to $\CA_G$, and define $\delta_{dR}^*$ by $(f\omega^I\delta_{dR}, \delta_{dR}^*)=\epsilon (f\omega^I)$. 

\begin{Lem}\label{Lem:CAGbasis}
Multiplication of elements of the form
\be
\varphi\cdot \omega_I,\qquad \varphi\cdot\omega_I\cdot \delta_{dR}^*
\ee
inside $\CA_G^*$ spans this topological vector space and identifies it with $\Dist (G)\otimes \bigwedge\!{}^*\fg[1]\otimes \bigwedge\!{}^* \delta_{dR}^*$. 

\end{Lem}

\begin{proof}

Let us fix a set of basis $\{f_i\}\in \CO_G$ with dual basis $\{f^i\}\in \Dist (G)$, basis $\{\omega^I\}\in \bigwedge \!{}^*\fg^*[-1]$ with dual basis $\{\omega_I\}\in \bigwedge \!{}^*\fg[1]$. We simply need to show that the following two sets
\be
 \{f_i\cdot \omega^I, f_i\cdot\omega^I\cdot \delta_{dR}\}\qquad \{f^i\cdot \omega_I, (-1)^{|f^i|+|\omega_I|}f^i\cdot\omega_I\cdot \delta^*_{dR}\}
\ee
form mutual dual basis under the Hopf pairing. 

It is clear that $f^i\cdot \omega_I$ pairs trivially with $f_i\cdot\omega^I\cdot \delta_{dR}$. It's pairing with $f_j\cdot \omega^J$ is given by
\be
(f_j\cdot \omega^J, f^i\cdot \omega_I)=(\Delta (f_j)\Delta (\omega^J), f^i\otimes \omega_I).
\ee
In order for this to be non-trivial, the only part of the coproduct of $\omega^J$ that can contribute is $1\otimes \omega^J$. Therefore the pairing is
\be
(f_j, f^i)(\omega^J, \omega_I)=\delta_{j}^i\delta^J_I. 
\ee
Similarly, $ f_j\cdot\omega^J\cdot \delta_{dR}$ pairs non-trivially only with $f^i\cdot\omega_I\cdot \delta^*_{dR}$, and their pairing is
\be
(\Delta (f_j\cdot\omega^J\cdot \delta_{dR}), f^i\cdot \omega_I\otimes \delta^*_{dR}).
\ee
For this to be non-trivial, the only part of the coproduct of $\delta_{dR}$ that contributes is $1\otimes \delta_{dR}$. Similarly, the only term in the coproduct of $\Delta (\omega^J)$ that can contribute is $\rho (\omega^J)^{op}$. The above then becomes
\be
(-1)^{|f^i|+|\omega_I|}(f_j\cdot\omega^J,  f^i\cdot \omega_I)=(-1)^{|f^i|+|\omega_I|} \delta_{j}^i\delta^J_I. 
\ee
This completes the proof.

\end{proof}

We now calculate the commutation relations of the generators using the above basis. 

\begin{Lem}
The element $\delta_{dR}^*$ is central and primitive. The subspace spanned by $f^i\cdot \omega_I$ is identified with the algebra $\Dist (T[1]G)$. 
\end{Lem}

\begin{proof}
The fact that $f^i\cdot \omega_I$ span $\Dist (T[1]G)$ is clear, since it is identical to the dual calculation of $\Omega_G$. We first show that $\delta_{dR}^*$ is central.  Let $\varphi\in \Dist (T[1]G)$ and $f\in \Omega_G$, then 
    \be
(f\cdot \delta_{dR},   \varphi\cdot \delta_{dR}^*)=(\Delta (f)\Delta (\delta_{dR}), \varphi\otimes \delta_{dR}^*)=(\delta_{dR}, \delta_{dR}^*)(-1)^{|\varphi|}(f, \varphi)=(-1)^{|\varphi|}( f\cdot \delta_{dR},  \delta_{dR}^*\cdot \varphi),
    \ee
    therefore $\delta_{dR}^*$ commutes with all elements in $\Dist (T[1]G)$. Clearly $(\delta_{dR}^*)^2=0$ so $\delta_{dR}^*$ is central.  To show that $\delta_{dR}^*$ is primitive, we simply note that for any $f_1,f_2\in \Omega_G$, the element $\delta_{dR}^*$ pairs non-trivially only with $ f_1f_2\delta_{dR}$ or $f_1\delta_{dR} f_2$, and the two pairings are both equal to $\epsilon (f_1f_2)$, since $(f_1[\delta_{dR}, f_2], \delta^*_{dR})=0$.

\end{proof}

The coproduct on $\Dist (T[1]G)$ can now be easily calculated. Let $\varphi\in \Dist (T[1]G)$, we can write
\be
\Delta_\delta (\varphi)=\varphi_1+\delta^*_{dR}\otimes 1(\varphi_2)+1\otimes \delta^*_{dR}(\varphi_3)+ \delta^*_{dR}\otimes \delta_{dR}^*(\varphi_4).
\ee
It is easy to see that $\varphi_3=\varphi_4=0$ and $\varphi_1=\Delta (\varphi)$. To calculate $\varphi_2$, we have
\be
(f_1\delta_{dR}\otimes f_2, \Delta_\delta (\varphi))=(f_1\delta_{dR}f_2, \varphi)=(f_1 [\delta_{dR}, f_2], \varphi). 
\ee
If $\varphi\in \Dist (G)$ then the above is zero, so $\Delta_\delta=\Delta$ on $\Dist (G)$. Let now $\varphi=b\in \fg[1]$, then we must have $f_i\in \CO_G$ and the above is equal to
\be
f_1(e)\xi^R_b(f_2)(e)=(f_1\delta_{dR}\otimes f_2, \delta_{dR}^*\otimes \mathrm{Lie}_b),
\ee
and therefore $\Delta_\delta (b)=\Delta (b)+ \delta_{dR}^*\otimes \mathrm{Lie}_b$. 

The double $D(\CA_G)$ is the vector space $\CA_G\otimes \CA_G^{*, cop}$, whose coproduct is the coproduct of $\CA_G$ and opposite coproduct of $\CA_G^{*}$. We now compute the commutation relation. 

\begin{Prop}\label{Prop:DAGAG/G}

    The algebra $D(\CA_G)$ is a central extension of $A_{G/G_{ad}}$ by $\delta_{dR}^*$. The central extension is defined by the following commutation relation
    \be
[b, f]=\xi_b^L (f) \delta_{dR}^*. 
    \ee
    
\end{Prop}

\begin{proof}
    Since the coproduct of $\Dist (G)$ is unchanged, its commutation relation with $\CA_G$ follows from the commtutation relation in $D(\Omega_G)\subseteq A_{G/G_{ad}}$. The fact that $\delta_{dR}^*$ is central follows from the coproduct of $\CA_G$ and that $\delta_{dR}^*$ is primitive. Since $\Delta_\delta (b)$ is equal to $\Delta (b)+ \delta_{dR}^*\otimes \mathrm{Lie}_b$, and the correction term pairs trivially with coproducts of $c\in \fg^*[-1]$, the commutation relation between $b$ and $c$ is simply given by its commutation relation inside $D(\Omega_G)\subseteq A_{G/G_{ad}}$. 
    
    Let us now consider $f\in \CO_G$. Using the product rule of the double \eqref{eq:doubleprod}, we find
    \be
b\cdot f= fb+\delta^* f^{(2)}(f^{(1)}, S\mathrm{Lie}_b)=fb-\delta^*_{dR}\xi_b^L (f). 
    \ee
    Finally we have
    \be
    b\cdot \delta=-\delta\cdot b-\mathrm{Lie}_b(\delta_{dR},\delta^*_{dR}),
    \ee
    implying that $\{\delta, b\}=-\mathrm{Lie}_b$. Sending $b\mapsto -b$ gives the desired identification. 

\end{proof}

\begin{Cor}
The $R$-matrix of $D(\Omega_G)$ gives rise to an $R$-matrix of $A_{G/G_{ad}}$. 
\end{Cor}

\begin{proof}

Let $\CR_\CA$ be the $R$-matrix of $D(\CA_G)$. Using the Hopf pairing, we can write it as
\be
\CR_\CA= \sum (-1)^{|\omega_I|}f_i\omega^I\otimes f^i\omega_I+ \sum (-1)^{|\omega_I|+1}(-1)^{|\omega_I|}f_i\omega^I\delta_{dR}\otimes f^i\omega_I\delta_{dR}^*=\CR_\Omega\cdot \exp (-\delta_{dR}\otimes \delta^*_{dR}). 
\ee
Sending $\delta_{dR}^*\to 0$ proves the claim. 

\end{proof}

\begin{Rem}
Alternatively, one just need to show that $[\Delta (\delta_{dR}), \CR_\Omega]=0$ in $A_{G/G_{ad}}$, which is equivalent to the condition that $(d_{dR}f, \varphi)=-(f, d_{dR}\varphi)$. 

\end{Rem}

\begin{Rem}\label{Rem:twist}

The twist $\theta$ of $D(\Omega_G)$ (whose existence is guaranteed by Proposition \ref{Prop:twist} applied to the graded group $T[1]G$) gives a twist for $A_{G/G_{ad}}$, since it commutes with $\delta_{dR}$.

\end{Rem}

At the moment the $R$-matrix $\CR_\Omega$ is still not explicit as it is given by an infinite sum of tensor products. However, it can be written in terms of the following Yetter-Drinfeld data (as in \cite[Section 8.4.1]{DN24}). Let $M, N$ be two modules of $A_{G/G_{ad}}$. In particular, they are both modules of $\Omega_G$, together with compatible comodule structures $\rho_M, \rho_N$, which are maps
\be
\rho_M: M\to \Omega_G\otimes M, ~\rho_N: N\to \Omega_G\otimes N. 
\ee
The action of $\CR_\Omega$ is the composition
\be\label{eq:RYD}
\btik
M\otimes N \rar{1\otimes \rho_N} & M\otimes \Omega_G\otimes N \rar{\sigma_{12}} & \Omega_G\otimes M\otimes  N\rar{\mathrm{act}_M} & M\otimes N. 
\etik
\ee
Similarly, the twist acts as
\be
\btik
M\rar{\rho_M} & \CO_G\otimes M\rar{S\otimes 1} & \CO_G\otimes M\rar{\mathrm{act}_M} & M.
\etik
\ee

\subsubsection{Hopf subalgebra and character sheaf}\label{subsubsec:CSH}

The Hopf algebra $D(\CA_G)$ (as well as $A_{G/G_{ad}}$ and $D(\Omega_G)$) contains a Hopf subalgebra generated by $\Dist (G)$ and $\CO_G$. From the construction, it is clear that this Hopf subalgebra is identified with the double $D(\CO_G)$ of $\CO_G$. As a double, it comes equipped with an $R$-matrix $\CR_G$, which is the identity map on $\CO_G$. The double is a semi-direct product $\Dist (G)\ltimes \CO_G$, and so a module of $D(\CO_G)$ is nothing but a module of $\CO_G$ that is equivariant with respect to conjugation $G$-action. Namely we have an equivalence
\be
D(\CO_G)\Mod\simeq \QCoh (G/G_{ad}).
\ee
The stack $G/G_{ad}$ is the stack of loop spaces $\mathrm{Maps}(S^1, BG)$ by \cite{ben2010integral}. In this work, the authors showed that the DG category $\QCoh (G/G_{ad})$ is the Drinfeld center (in the appropriate DG context) of $\QCoh (BG)=\mathrm{Rep}(G)$. The element $\CR_G$ provides a computable braiding, which can also be expressed in terms of Yetter-Drinfeld data, as in equation \eqref{eq:RYD}. 

As an element in $D(\CA_G)$, $\CR_G$ is not a braiding any more, but it still satisfies the cocycle condition. The following twist equation follows from the Yang-Baxter equation of $\CR_G$:
\be
\CR_G^{12}\Delta_\delta\otimes 1(\CR_G)=\CR_G^{12}\CR_G^{13}\CR_G^{23} = \CR_G^{23}\CR_G^{13}\CR_G^{12}=\CR_G^{23}1\otimes \Delta_\delta (\CR_G).
\ee
In particular, $\CR_G$ can be used to twist the Hopf algebra $D(\CA_G)$. Similarly, $\CR_{G}^{21, -1}$ is also a twist. Using this element $\CR_G$, we can further decompose $\CR_\CA$ as
\be
\CR_\CA=\CR_G\cdot \exp (-\sum c^i\otimes b_i)\exp (-\sum \delta_{dR}\otimes \delta^*_{dR}). 
\ee
From the decomposition, we see that if we restrict this braiding to the heart of the standard $t$-structure of $\CD_\hbar (G/G_{ad})^{\C^\times}$, namely to $\CM_\hbar (G/G_{ad})^{\C^\times}$, we are left with $\CR_G$, which is precisely the braiding constructed in \cite{BFOcharacter}. 

\subsubsection{The algebra $H_G$}

The subalgebra of $D(\CA_G)$ generated by $\Dist (G), \fg^*[-1]$ and $\delta_{dR}$ is easily seen to be isomorphic to $H_G$ (Example \ref{exp:G}). This is not an embedding of Hopf algebras. In the last section, we show that one can perform a Hopf twist to turn $D(\CA_G)$ into $D(H_G)$. 

\subsubsection{Cliffold algebra and matrix factorizations}

Let us now consider the quotient $A_{G/G_{ad}}$. In this quotient, we can consider the Hopf subalgebra generated by $\CO_G, \fg[1]$ and $\fg^*[-1]$. Let us denote this by $\mathrm{Cl}(G, \fg)$. The reason for this notation is that generators $\fg[1]$ and $\fg^*[-1]$ obey a Clifford algebra relation over $\CO_G$. Indeed, if we consider the vector bundle $G\times (\fg\oplus \fg^*)$ over $G$, and the quadratic function $W=((g^{-1}-1)b, c)$ for $(b, c)\in \fg\oplus \fg^*$, then
\be
\{b_i, c^j\}= \frac{\pd^2 W}{\pd b_i \pd c^j}.
\ee
This algebra appears in the study of topological gauge theories in 3d \cite[Section 9.6.2]{DN24}, and it was argued there that the derived category of $\mathrm{Cl}(G, \fg)\Mod^G$ ($G$-equivariant modules) should be equivalent to the category $\MF_G (G\times T^*\fg, W)$, the category of $G$-equivariant matrix factorizations. From the previous sections, the element $\delta$ gives rise to a deformation of this category compatible with the braided tensor structure. 

\subsection{Double of $H_G$}\label{subsec:doubleHG}

We will use the same method to calculate the double of $H_G$. Let us identify $H_G$ with $\bigwedge\!{}^* \fg^*[-1]\otimes \C[\delta]\otimes \Dist (G)$ via multiplication. Denote by $\omega^\vee\in \bigwedge\!{}^* \fg[1]$ the element in $H_G^*$ such that
\be
(\omega\cdot \varphi, \omega^\vee)=(\omega, \omega^\vee)\epsilon (\varphi), \qquad (\omega \cdot  \delta\cdot \varphi, \omega^\vee)=0. 
\ee
One can similarly define $f\in \CO_G$ and $\delta^*$ as elements in $H_G^*$. The following is true, whose proof is almost identical to that of Lemma \ref{Lem:CAGbasis}.

\begin{Lem}
The following are a set of dual basis
\be
\{\omega^If^i, \omega^I\delta f^i\},\qquad \{\omega_If_i, (-1)^{|\omega_I|}\omega_I\delta^*f_i\}
\ee
in $H_G$ and $H_G^*$.

\end{Lem}

\begin{proof}
The proof is identical to Lemma \ref{Lem:CAGbasis}, except that now we use the fact that in the coproduct of $\delta$, elements in $\Dist (G)$ appears to the left of $\bigwedge \!{}^*\fg^*[-1]$. This guarantees that $\omega^I\delta f^i$ pairs trivially with $\omega_If_i$. 

\end{proof}

We can now compute the commutation relation of $H_G^*$. 

\begin{Lem}

In $H_G^*$, the element $\delta^*$ is both central and primitive. Moreover, we have
\be
[b, f]=-\xi_b^L(f)\delta^*, \qquad \forall b\in \fg[1], ~f\in \CO_G. 
\ee

\end{Lem}

\begin{proof}
Primitivity of $\delta^*$ follows from the fact that commutator with $\delta$ increases the wedge degree of $\omega^I$, and therefore any term involving a commutator with $\delta$ will pair trivially with $\delta^*$. To prove that $\delta^*$ is central, first note that $(\delta^*)^2=0$. For its commutator with $\bigwedge \fg[1]$, we have
\be
(\omega^J\cdot f^j, \delta^*\cdot \omega_I)=0= (\omega^J\cdot f^j,, \omega_I\cdot \delta^*),~(\omega^J\cdot\delta\cdot f^j, \delta^*\cdot \omega_I)=(-1)^{|\omega_I|}\delta_I^J\epsilon (f^j)= (\omega^J\cdot\delta\cdot f^j, \omega_I\cdot \delta^*). 
\ee
A similar result holds for $f\in \CO_G$. For the commutator between $b$ and $f$, we have
\be
(\omega\cdot \varphi, f\cdot b)=(\Delta (\omega)\Delta (\varphi), f\otimes b)=(\varphi, f)(\omega, b)=(\omega\cdot \varphi, b\cdot f),
\ee
whereas
\be
(\omega\cdot \delta\cdot \varphi, f\cdot b)=(\Delta (\omega)\Delta (\delta)\Delta (\varphi), f\otimes b).
\ee
For this to be non-zero, the only contribution from $\Delta (\delta)$ has to be $\sum x_i\otimes c^i$, and therefore $\omega=1$. The result is given by
\be
(x_i\varphi, f)(c^i, b)=(\varphi, \xi_b^L(f)). 
\ee
This implies that $[b, f]=-\xi_b^L(f)\delta^*$ as desired.

\end{proof}

It is clear that the coproduct on $\CO_G\subseteq H_G^*$ is given by the standard coproduct of $\CO_G$. We now compute the coproduct of $\fg[1]$.

\begin{Lem}
Let $b\in \fg[1]$, then 
\be
\Delta (b)=b\otimes 1+ S\otimes 1\rho^\vee(b),
\ee
where $\rho^\vee: \fg\to \CO_G\otimes \fg$ is the coaction of $\CO_G$ on $\fg$. 

\end{Lem}

\begin{proof}

First of all, $\Delta (b)$ does not involve $\delta^*$. This is again because commutator with $\delta$ increases the wedge degree of $\fg^*[-1]$. For any $\varphi_i, \omega_i$ and $i=1, 2$, we have
\be
(\omega_1\varphi_1\otimes \omega_2\varphi_2, \Delta (b))=(\omega_1\varphi_1\omega_2\varphi_2, b)=(\omega_1\cdot \varphi_1\rhd\omega_2, b)\epsilon (\varphi_2). 
\ee
This proves the claim.

\end{proof}

With the calculations of the coproduct, we can now compute the double as an algebra. 

\begin{Prop}\label{Prop:DHDAiso}
The algebra $D(H_G)$ is canonically isomorphic to $D(\CA_G)$, by identifying the generators
\be
\delta\to \delta_{dR}, ~b\mapsto b, ~c\mapsto c, ~\delta^*\mapsto \delta^*,
\ee
as well as $\Dist (G)$ and $\CO_G$.

\end{Prop}

\begin{proof}
This is a purely computational result. Clearly $\delta^*$ is central. We compute the rest of the commutation relations. 

Let us first consider $f\cdot \delta$. We have
\be
f\cdot \delta= (Sf^{(3)}, \delta^{(1)})\delta^{(2)}f^{(2)}(f^{(1)}, \delta^{(3)}).
\ee
Using the coproduct of $\delta$, this has two contributions of the form
\be
\delta \cdot f+(Sf^{(2)}, x_i)c^if^{(1)}=\delta \cdot f-c^i \xi_{x_i}^R(f),
\ee
and we find $[\delta, f]=[\delta_{dR}, f]$. Similarly, the double coproduct of $b$ is of the form
\be
1\otimes \Delta^{op}\Delta^{op}(b)=1\otimes 1\otimes b+ 1\otimes 1\otimes S\rho^\vee(b)^{op}+1\otimes \Delta^{op}(1\otimes S\rho^\vee(b)^{op}).
\ee
The only two terms that can pair non-trivially with each other are $1\otimes \delta\otimes 1$ with $1\otimes 1\otimes S\rho^\vee(b)^{op}$ and $1\otimes x_i\otimes c^i$ with $1\otimes 1\otimes S\rho^\vee(b)^{op}$, since in other terms in the coproduct of $\delta$ elements in $\Dist (G)$ appears to the left of $\bigwedge \!{}^*\fg^*[-1]$. Using this, we find that
\be
b\cdot \delta=-\delta\cdot b-x_i (c^i, b)=-\mathrm{Lie}_b. 
\ee
Finally, one can compute the commutator of $b, c$ which gives
\be
b\cdot c=-c\cdot b-(c, b)+((-)^{-1}b, c),
\ee
as desired. The commutation relations of $\varphi$ with other generators are clear. 

\end{proof}

\subsubsection{Shifted Poisson Lie groups and deformations}\label{subsubsec:ShiftedPoisson}

The Hopf subalgebra $H^*_G\subseteq D(H_G)$ is generated by $\CO_G, \fg[1]$ and $\delta^*_{dR}$ with commutation relations
\be
[b, f]=-\xi_b^L(f)\delta^*_{dR},\qquad b\in \fg[1], f\in \CO_G. 
\ee
Note that if we set $\delta^*_{dR}=0$, then the quotient $H^*_G/(\delta^*_{dR})$ can be identified with functions on $T^*[-1]G$. This is in fact an identification of Hopf algebras, if we identify $T^*[-1]G=G\ltimes \fg^*[-1]$ as the group extension of $G$ by $\fg^*[-1]$. It's Lie algebra is precisely $T^*[-1]\fg$, which was used to construct $H_\fg$. The 1-shifted Lie bialgebra structure on $T^*[-1]\fg$ gives rise to a 1-shfited Poisson structure $\{-, -\}$ on $T^*[-1]G$, making it a 1-shifted Poisson Lie group. This is very similar to the fact that a Lie bialgebra structure on $\fg$ is related to a Poisson Lie group structure on $G$. We see that $H^*_G$ can be viewed as a quantization, in the sense that
\be
[b, f]=-2\xi_b^L (f)\delta^*=-2\{\xi_b^L, f\}\delta^*,
\ee
This is analogous to the fact that the quantization of a Lie bialgebra provides a quantization of the dual Poisson Lie group. 

\subsection{Identifying the doubles}

We finish by proving the second part of Theorem \ref{Thm:double}. Let $\CF=\CR_G^{21, -1}$ be the twist, and denote by $D(\CA_G)^{\CF}$ the twist of $D(\CA_G)$ by $\CF$. 

\begin{Prop}
    The embedding $H_G\to D(\CA_G)^{\CF}$ is a map of Hopf algebras.
\end{Prop}

\begin{proof}
    We calculate $\CF\Delta_\delta \CF^{-1}$ explicitly and show that it agrees with $H_G$. This is clear for $\Dist (G)$ as it is cocommutative in $D(\CA_G)$ as well as $D (\CO_G)$. Let $c\in \fg^*[-1]$, we have
    \be
\CF\Delta_\delta (c)\CF^{-1}= 1\otimes S(\CR_{G}^{21}) \Delta_\delta (c)\CR_G^{21}. 
    \ee
   Recall that $\Delta_\delta (c)=1\otimes c+\rho (c)^{op}$ from Corollary \ref{Cor:coprodCAG}. Moreover, since $\CR_G^{21}$ commutes with $1\otimes c$, we simply need to show that
   \be
\CR_G^{21} (c\otimes 1)\CR_G^{21, -1}=\rho (c)^{op}. 
   \ee
   Both being elements in $D(\CA_G)\otimes \CO_G$, we pair them with an element $g\in G$. The RHS is $g\rhd c$, the adjoint action of $g$ on $c$. The LHS is equal to $g cg^{-1}$ computed in $D(\CA_G)$, which coincides with $g\rhd c$ thanks to Proposition \ref{Prop:DAGAG/G}. Therefore $\CF\Delta_\delta (c)\CF^{-1}=c\otimes 1+1\otimes c$. Consider now
   \be
   \CF\Delta_\delta (\delta_{dR})\CF^{-1}=1\otimes S(\CR_{G}^{21}) (\delta_{dR}\otimes 1+1\otimes \delta_{dR})\CR_G^{21}.
   \ee
   Again $\CR_G^{21}$ commutes with $\delta_{dR}\otimes 1$, we simply need to show that
   \be
1\otimes S(\CR_{G}^{21})(1\otimes \delta_{dR}) \CR_G^{21}=1\otimes \delta_{dR}+ \sum  x_i\otimes c^i,
   \ee
or equivalently, $[\CR_G^{21, -1},1\otimes \delta_{dR}]\CR_G^{21}= \sum x_i\otimes c^i$. Let us write $\CR_G=\CR^{(1)}\otimes \CR^{(2)}$, then 
\be
[\CR_G^{21, -1},1\otimes \delta_{dR}]\CR_G^{21}= -S(\CR^{(2)})\CR^{(2)'}\otimes \delta_{dR}(\CR^{(1)})\CR^{(1)'}= -S(\CR^{(2)})\CR^{(2)'}\otimes x_i(\CR^{(1)})\CR^{(1)'}c^i. 
\ee
The element $-S(\CR^{(2)})\CR^{(2)'}\otimes x_i(\CR^{(1)})\CR^{(1)'}$ belongs to $D(\CA_G)\otimes \CO_G$, we pair them with $g\in G$, which gives
\be
-S(gx_i)g=x_i g^{-1}g=x_i.
\ee
Therefore 
\be
 -S(\CR^{(2)})\CR^{(2)'}\otimes x_i(\CR^{(1)})\CR^{(1)'}=\sum x_i\otimes c^i,
\ee
as desired. 

\end{proof}

We can now finish the proof of Theorem \ref{Thm:double}. One could perform a similar calculation for the rest of the generators, but we prefer to use the polarization of $R$-matrix to achieve this. 

\begin{proof}[Proof of Theorem \ref{Thm:double}]

We show that $D(H_G^*)\cong D(\CA_G)^{\CF}$ using the polarization of the $R$-matrix. The $R$-matrix of $D(\CA_G)^{\CF}$ is 
\be
\CR_\CF= \CF^{21}\CR_\CA \CF^{-1}= \CR_G^{-1} \CR_\CA \CR_G^{21}. 
\ee
We have seen that $\CR_\CA$ can be written as $\CR_G\exp (-\sum c^i\otimes b_i)\exp (-\delta_{dR}\otimes\delta_{dR}^*)$, therefore
\be
\CR_\CF=\exp (-\sum c^i\otimes b_i)\exp (-\delta_{dR}\otimes\delta_{dR}^*)\CR_G^{21}.
\ee
In particular, $\CR_\CF\in  H_G\wh \otimes D(\CA_G)^{\CF}$. The polarization of the $R$-matrix implies that there exists a map 
\be
D(H_G)=H_G\otimes H_G^{*, cop}\longrightarrow D(\CA_G)^{\CF}.
\ee
Indeed, the map $H_G^{*, cop}\to D(\CA_G)^{\CF}$ is supplied by $\CR_\CF$ by $\CR_\CF (a\otimes 1)$ for $a\in H_G$. The fact that this is a map of Hopf algebras, and the fact that the images of $H_G^{*, cop}$ and $H_G$ obey the commutation relation of the double follow from the cocycle conditions of $\CR_\CF$. Moreover, $\CR_\CF$ can be viewed as an element in $H_G\otimes \overline{H_G^{*, cop}}$, where $ \overline{H_G^{*, cop}}$ is the image of $H_G^{*, cop}$ in $D(\CA_G)^{\CF}$. It is now clear from the calculation of $D(H_G)$ and its pairing in Section \ref{subsec:doubleHG} that the map induecd from $\CR_\CF$ is precisely the identification of Proposition \ref{Prop:DHDAiso}. We have now completed the proof.

\end{proof}

\appendix

\section{Drinfeld double of an infinite-dimensional Hopf algebra}\label{app:double}

In this appendix, we consider infinite-dimensional graded Hopf algebras and their double. This is to provide the necessary rigorous background for the constructions appearing in the main body of the paper. 

\subsection{Dual of an infinite dimensional Hopf algebra}

Let $H$ be a graded Hopf algebra with invertible antidote $S$, and assume that $H$ is concentrated in finitely many cohomological degrees. All the operations we will perform respect this grading. Let $H^*$ be the linear dual of $H$. It admits a pro-finite topology from finite-dimensional filtrations of $H$, and also admits a grading thanks to the fact that $H$ is concentrated in finitely many degrees. All the operations we define will respect this grading. 

Define the following maps $\nabla: H^*\wh\otimes H^*\to H^*$ and $\Delta: H^*\to H^*\wh\otimes H^*$ by
\be
h(a, \nabla (\alpha\otimes \beta))=h(\Delta(a), \alpha\otimes \beta), \qquad h(a\otimes b, \Delta (\alpha))=h(ab, \alpha),\qquad a,b\in H, \alpha,\beta\in H^*. 
\ee
The first map does extend to the completion since $H^*\wh\otimes H^*=(H\otimes H)^*$. 

\begin{Lem}
    The maps $\Delta, \nabla$ makes $H^*$ into a pro-finite Hopf algebra. 
\end{Lem}

\begin{proof}

    The statement of the lemma means that tensor products of $H^*$ are taken using the pro-finite topology of $H^*$. This is a simple consequence of dualizing the Hopf structure of $H$, taking into account of the fact that $(H^*)^{\wh\otimes n}=(H^{\otimes n})^*$. 
    
\end{proof}

A smooth module of $H^*$ is  a vector space $M$ with a map $\rho: H^*\to \End (M)$ such that the following two conditions hold.

\begin{enumerate}

    \item For any $m\in M$ the map $\rho (-)m$ factors through a finite-dimensional quotient of $H^*$.

    \item For any $\alpha, \beta\in H^*$, $\rho(\alpha\cdot \beta)m=\rho(\alpha)\rho (\beta)m$. 
    
\end{enumerate}

Note that these conditions do not imply that there exists a map $H^*\wh\otimes M\to M$, unless $M$ is finite-dimensional. However, they do imply that $M$ is filtered by finite-dimensional submodules, which are obviously smooth. By $H^*\Mod$ we always mean the category of smooth modules. 

\begin{Lem}

   The category $H^*\Mod$ is a monoidal category whose rigid objects are finite-dimensional modules. 
   
\end{Lem}

\begin{proof}
    Let $M, N$ be smooth modules. We claim that $\Delta$ gives a well-defined smooth module structure on $M\otimes N$. Without loss of generality, we assume that $M, N$ are finite-dimensional, and therefore their tensor product is, and hence it is clear that $\Delta$ gives a well-defined action. The rest of the statements are clear. 
    
\end{proof}

The relation between modules and comodules remain true in the setting of infinite dimensional Hopf algebras.

\begin{Prop}
    There is an equivalence of monoidal categories
    \be
H^*\Mod\simeq H\mathrm{-coMod}.
    \ee
\end{Prop}

\begin{proof}
Objects in LHS and RHS are both filtered by finite-dimensional subobjects \cite[Section 4.c]{milne2017algebraic}, so we just need to prove the equivalence for finite-dimensional objects. Let $V$ be a finite-dimensional vector space. A comodule structure on $V$ is a map
\be
(V\to H\otimes V)\equiv (V\otimes V^*\to H)\equiv (H^*\to \End (V)).
\ee
It is easy to see that the resulting map $H^*\to \End (V)$ gives a module structure, and vice versa. Here we use the fact that $\End (V)$ is finite-dimensional. The statement about the monoidal structure is clear.

\end{proof}

\subsection{Doubling an infinite-dimensional Hopf algebra}\label{subsec:doubleappendix}

We now consider the double of $H$. Let $D(H):=H\otimes H^*$. An element in this is a finite sum of elements of the form $a\otimes \alpha$. It is important that we only allow finite sums. Since we are dealing with graded Hopf algebras, signs in the definition of the double are important. This was carefully treated in \cite{gould1993quantum} and \cite[Section 5.1]{DN24}, as we recall now. Define an algebra structure on $D(H)\otimes D(H)$ by
\be\label{eq:doubleprod}
(a\otimes \alpha)\cdot (b\otimes \beta)=(-1)^{\xi(\alpha, b)} (S^{-1}b^{(1)}, \alpha^{(1)})ab^{(2)}\otimes \alpha^{(2)}\beta(b^{(3)}, \alpha^{(3)})
\ee
for $a,b\in H$ and $\alpha, \beta\in H^*$. Here the sign is given by
\be
\xi(\alpha, a)= (\sum_{i\leq n\leq 3}|\alpha^{(i)}|)|a^{(i)}|.
\ee

\begin{Lem}
    Equation \eqref{eq:doubleprod} provides a well-defined algebra structure on $D(H)$. 
\end{Lem}

\begin{proof}
    The well-definedness follows from the fact that for any $b\in H$, the triple coproduct $1\otimes \Delta\Delta (b)$ belong to a finite-dimensional subspace of $H^{\otimes 3}$, and therefore the sum
    \be
(-1)^{\xi(\alpha, b)} (S^{-1}b^{(1)}, \alpha^{(1)})b^{(2)}\otimes \alpha^{(2)}(b^{(3)}, \alpha^{(3)})
    \ee
    is finite. The fact that this is a well-defined algebra structure can be checked similar to \cite{gould1993quantum}. 
    
\end{proof}

The coproduct $\Delta:=\Delta_H\otimes \Delta_{H^*}^{op}$ can still be defined, but it has the annoying feature that it maps $H\to H\otimes H$ but $H^*\to H^*\wh\otimes H^*$.  However, this can be overcome by noticing that the image of $\Delta\otimes \Delta^{op}$ can be naturally identified with $D(H^{\otimes 2})$, which has an algebra structure defined again via equation \eqref{eq:doubleprod}. 

\begin{Prop}
    The map $\Delta$ makes $D(H)$ into a quasi-triangular Hopf algebra, whose $R$-matrix belongs to an appropriate algebra closure of $D(H^{\otimes 2})$:
    \be
\CR=\sum (-1)^{|f_i|} f_i\otimes f^i, \qquad \{f_i\}\subseteq H,~\{f^i\}\subseteq H^*, ~(f_i, f^j)=\delta_i^j.
    \ee
    In particular, the category of modules of $D(H)$ that are smooth as $H^*$ modules has the structure of a braided tensor category. 
    
\end{Prop}

\begin{proof}

The standard arguments for the quantum double of finite-dimensional Hopf algebras \cite{gould1993quantum} applies here without change, since all calculations are well-defined, except statements about the $R$-matrix. This is because $\CR\in H\wh\otimes H^*$ rather than $D(H^{\otimes 2})$, so we focus on this part. We need to show that $\CR$ satisfy cocycle condition and that $\CR\Delta\CR^{-1}=\Delta^{op}$. The cocycle conditions are
\be\label{eq:cocycleinf}
\Delta\otimes 1\CR=\CR^{13}\CR^{23}\in (H\otimes H)\wh\otimes H^*, \qquad 1\otimes \Delta (\CR)=\CR^{13}\CR^{12}\in H\wh \otimes (H^*\wh\otimes H^*).
\ee
Assuming that $f^if^j=\sum c^{ij}_k f^k$, then 
\be
\CR^{13}\CR^{23}=\sum (-1)^{|f_i|+|f^i||f_j|+|f_j|} f_i\otimes f_j\otimes f^if^j=\sum (-1)^{|f_i|+|f^i||f_j|+|f_j|}c^{ij}_k f_i\otimes f_j\otimes f^k. 
\ee
This is well-defined in $(H\otimes H)\wh\otimes H^*$ since for each $k$ there are only finitely many nonzero $c^{ij}_k$, thanks to the fact that
\be
c^{ij}_k=(f_k, f^if^j)=(\Delta (f_k), f^i\otimes f^j)~~\equiv~~\Delta (f_k)=\sum_{i,j} c^{ij}_k (-1)^{|f_i||f_j|} f_i\otimes f_j. 
\ee
Moreover, the above implies
\be
\Delta\otimes 1(\CR)=\sum (-1)^{|f_i||f_j|+|f_k|} c^{ij}_k f_i\otimes f_j\otimes f^k. 
\ee
The first equality of \ref{eq:cocycleinf} now follows since $|f_k|=|f_i|+|f_j|$. The second equality of \ref{eq:cocycleinf} can be proven similarly. 

Finally, we need to make sense of $\CR\Delta \CR^{-1}$, because in apriori $\CR$ does not belong to $H\otimes H^*$. For this, we simply restrict to a smooth module of $D(H)$. On such a module, $\CR$ is well-defined, and so is $\CR\Delta \CR^{-1}$. It now follows from the standard quantum double calculations  \cite{gould1993quantum} that $\CR\Delta \CR^{-1}=\Delta^{op}$ indeed holds on smooth modules. This completes the proof. 

\end{proof}

We have now defined the quantum double of $H$, which is $D(H)$ with a slight generalization of the quasi-triangular structure. One can obviously also defined the double of $D(H^*)$ to be $H^*\otimes H$ with a similar quasi-triangular structure. Of course $D(H^*)$ and $D(H)$ can be twisted into each other by the $R$-matrix. The $R$-matrix can be alternatively defined by the following map on smooth modules:
\be
\btik
M\otimes N\rar{\rho_N} & M\otimes H\otimes N\rar{\sigma^{12}} &H\otimes  M\otimes N\rar{\mathrm{act}_M} & M\otimes N
\etik
\ee
where $\sigma$ is the flip map. It is now rather straightforward to see that a smooth module of $D(H)$ is simply a module of $H$ together with a compatible comodule structure. The compatibility condition precisely leads to the category of Yetter-Drinfeld modules \cite[Section 7.15]{EGNO}. 

\begin{Prop}

    There is an equivalence of braided tensor categories
    \be
D(H)\Mod\simeq \CY\CD_H^H, 
    \ee
    where $\CY\CD_H^H$ stands for the category of Yetter-Drinfeld modules. 
    
\end{Prop}

\begin{Rem}

Note that we can't define the double to be $H\wh\otimes H^*$, since this might not have an algebra structure. On smooth modules, any element in this vector space still acts, but the problem is that one can't change the order of the multiplication. 

\end{Rem}

\subsection{Example: algebraic groups}\label{app:OG}

We consider the example $H=\CO_G$ for $G$ an affine algebraic group. All the statements here hold true for DG groups with identical proofs (except that one needs to treat the signs  carefully). Let $\Dist (G)=\CO_G^*$ be the dual Hopf algebra. The equivalence
\be
\Dist (G)\Mod\simeq \CO_G\mathrm{-coMod}\simeq \mathrm{Rep}(G)
\ee
implies that $\Dist (G)$ is the correct pro-finite version of the group algebra of $G$. Let $\C \cdot G$ be the group algebra of $G$ and $U(\fg)$ the universal enveloping algebra of $\fg=\mathrm{Lie}(G)$. Define maps
\be\label{eq:DistGgA}
\begin{aligned}
& g\in G\mapsto \delta_g (f)=f(g), ~~\forall f\in \CO_G\\
& X\in \fg\mapsto \pd_X (f)=\xi_X^Rf(e), ~~\forall f\in \CO_G
\end{aligned}
\ee

\begin{Lem}
The maps in equation \eqref{eq:DistGgA} are maps of Hopf algebras. 

\end{Lem}

\begin{proof}

Let us consider $\C\cdot G$. The statement for $U(\fg)$ is similar. We have
\be
\delta_{g}\cdot \delta_h (f)=\delta_g\otimes \delta_h (\Delta (f))=f(gh)=\delta_{gh}(f),
\ee
and
\be
(\Delta (\delta_g), f_1\otimes f_2)=(\delta_g, f_1\cdot f_2)=f_1(g)f_2(g)=(\delta_g\otimes \delta_g, f_1\otimes f_2),
\ee
as desired. 

\end{proof}

Note that the map $\C\cdot G\to \Dist (G)$ is dense, since for any $f\in \CO_G$, if $f(g)=0$ for all $g$ then $f=0$. In particular, the restriction functor
\be
\Dist (G)\Mod\longrightarrow \C\cdot G\Mod
\ee
is fully-faithful. The image is precisely $\Rep (G)$, the category of algebraic representations of $G$. 

Let us consider $D(\CO_G)=\CO_G\otimes \Dist (G)$. This now has the structure of a quasi-triangular Hopf algebra in the sense of Appendix \ref{subsec:doubleappendix}. 

\begin{Prop}

Conjugation action induces an action of $\Dist (G)$ on $\CO_G$ and $D(\CO_G)=\Dist (G)\ltimes \CO_G$. 

\end{Prop}

\begin{proof}

The simplest way to show this is to use the fact that $\C\cdot G\to \Dist (G)$ is dense. Using equation \eqref{eq:doubleprod}, we find the following relation in $D(\CO_G)$. 
\be
g\cdot f= f^{(1)}(g^{-1})f^{(2)}\cdot g f^{(3)}(g)=\mathrm{Ad}_g (f) g. 
\ee
Generally, let $\varphi\in \Dist (G)$, one can define
\be
\varphi\rhd f= f^{(1)}(S\varphi^{(1)}) f^{(2)}f^{(3)}(\varphi^{(2)}).
\ee
This is well-defined thanks to the pro-finiteness of the pairing, and is easily checked to define an action of $\Dist (G)$ on $\CO_G$. As we see, this is precisely the conjugation action.

\end{proof}

This is the statement that we have seen in Section \ref{subsubsec:CSH}, that a module of $D(\CO_G)$ is a quasi-coherent sheaf on $G/G_{ad}$. The braiding here is identical to the one defined by \cite{boyarchenko2014character} in the setting of quasi-coherent sheaves. 

\begin{Exp}
Assume that $G$ is reductive, then
\be
\CO_G=\bigoplus_{V\in G-\mathrm{Irred}} V\otimes V^*,\qquad \Dist (G)=\prod_{V\in G-\mathrm{Irred}} \End (V).
\ee
The action of $\Dist (G)$ on $\CO_G$ here is the obvious one. Multiplication on $\Dist (G)$ is the multiplication on $\End (V)$. The comultiplication maps $\varphi\in \End (V)$ to $\End (V_1\otimes V_2)$ for any $V\subseteq V_1\otimes V_2$. The antipode $S$ maps $\varphi$ to its conjugate, viewed as an element in $\End (V^*)$. For general $G$, one replaces the direct sum by an injective limit and the product by a projective limit.

\end{Exp}

The category of finite-dimensional modules $D(\CO_G)\Mod_f$ of $D(\CO_G)$ is a rigid braided monoidal category. Objects in this category are coherent sheaves supported at finite orbits of $G$. We show that the category of finite-dimensional modules admit a pivotal structure, via an explicit ribbon element. 

\begin{Prop}\label{Prop:twist}
The category $D(\CO_G)\Mod$ admits a twist given by
\be
\theta= \nabla 1\otimes S(\CR), \qquad \Delta (\theta)=(\CR^{21}\CR)^{-1}\theta\otimes \theta.
\ee
Consequently, the category of finite-dimensional modules admit a pivotal structure, by a work of Deligne (see \cite{yetter1993framed}). 

\end{Prop}

\begin{proof}

Although $\CR$ lives in a completion of $H\otimes H^*$, it is easy to see that $\theta$ is well-defined when acting on any smooth module of $D(\CO_G)$. On the other hand, the natural   element $u=\nabla S\otimes 1(\CR^{21})$ is not. Let us first show that $\theta$, when acting on smooth modules, indeed commutes with the action of $D(\CO_G)$. Let first $\varphi\in \Dist (G)$, we have
\be
 \varphi\theta=\varphi \CR^{(1)}S\CR^{(2)}=\varphi^{(1)}\rhd\CR^{(1)}\varphi^{(2)}S\CR^{(2)}.
\ee
This can be veiwed as an element in $\CO_G\wh\otimes \Dist (G)$ (functions on $G$ valued in $\Dist (G)$), we pair it with an element $g\in G$, which gives
\be
\CR^{(1)}(S\varphi^{(11)}g\varphi^{(12)})\varphi^{(2)} S\CR^{(2)}=\varphi^{(2)} S\left(S\varphi^{(11)}g\varphi^{(12)} \right)=\varphi^{(3)}S\varphi^{(2)}g^{-1}\varphi^{(1)}=g^{-1}\varphi. 
\ee
It is clear that $\theta\varphi$ is an element in $\CO_G\wh\otimes \Dist (G)$ with the same pairing with $g$. A proof for $f\in \CO_G$ is identical. 

We now need to verify the identity 
\be
\Delta (\theta)=(\CR^{21}\CR)^{-1}\theta\otimes \theta. 
\ee
Note that $\theta$, as an element in $\CO_G\wh\otimes \Dist (G)$, maps $g\to g^{-1}$. The element $\Delta (\theta)$ can be viewed as an element in $\CO_G^{\otimes 2}\wh\otimes \Dist(G)^{\wh\otimes 2}$ mapping $g\otimes h$ to $\Delta (h^{-1}g^{-1})=h^{-1}g^{-1}\otimes h^{-1}g^{-1}$. Let us show that the RHS has the same property. We have
\be
(\CR^{21}\CR)^{-1}\theta\otimes \theta = (S\CR^{(1)}\otimes \CR^{(2)})\cdot (\CR^{(2)'}\otimes S\CR^{(1)'})\cdot (\theta^{(1)}\otimes \theta^{(1)'})\cdot (\theta^{(2)}\otimes \theta^{(2)'})
\ee
in which $\theta^{(1)}\otimes\theta^{(1)'}\in \CO_G\otimes \CO_G$ and $\theta^{(2)}\otimes\theta^{(2)'}\in \Dist (G)\otimes \Dist (G)$. We need to commute $\CR^{(1)}$ pass $\theta^{(1)}$ in the above, to interprete this as an element in $\CO_G^{\otimes 2}\wh\otimes \Dist(G)^{\wh\otimes 2}$. This is easy to do, and is given by
\be
(\CR^{21}\CR)^{-1}\theta\otimes \theta =\lp S\CR^{(1)}\CR^{(21)'}\rhd \theta^{(1)}\otimes (\CR^{(21)}\rhd S\CR^{(1)'}\theta^{(1)'})\rp \cdot \lp \CR^{(22)'}\theta^{(2)}\otimes \CR^{(22)}\theta^{(2)'}\rp. 
\ee
Now we can pair this with an element of the form $g\otimes h$, which gives
\be
\begin{aligned}
&\lp (ghg^{-1})^{-1}\rhd \theta^{(1)}(g)\otimes \lp g^{-1}\rhd \theta^{(1)'}(h)\rp \rp \cdot \lp (gh^{-1}g^{-1})\theta^{(2)}\otimes g^{-1}\theta^{(2)'}\rp\\
& = \lp \theta^{(1)}(ghgh^{-1}g^{-1})\otimes \theta^{(1)'}(ghg^{-1})\rp \cdot \lp  (g^{-1}h^{-1}g)\theta^{(2)}\otimes g^{-1}\theta^{(2)'}\rp \\
&=(gh^{-1}g^{-1}) ghg^{-1}h^{-1}g^{-1}\otimes g^{-1} g h^{-1}g^{-1}=h^{-1}g^{-1}\otimes h^{-1}g^{-1}
\end{aligned}
\ee
This perfectly match the result from $\Delta (\theta)$. Therefore $\Delta (\theta)=(\CR^{21}\CR)^{-1}\theta\otimes \theta$, and the proof is complete.

\end{proof}

\begin{Rem}
    Note that the relation $S\theta=\theta$ does not hold. In fact, if the module $V$ of $D(\CO_G)$ is infinite dimensional, then $S\theta$ might not be well-defined. For finite-dimensional modules, the existence of $\theta$ allows one to define a different rigidity structure, so that $\theta_{V^*}=\theta_V^*$ does hold. 
\end{Rem}

\bibliographystyle{alpha}

\bibliography{HC}

@article{yetter1993framed,
	author = {Yetter, David N},
	journal = {Contemporary mathematics},
	pages = {325--325},
	publisher = {American Mathematical Society},
	title = {{Framed tangles and a theorem of Deligne on braided deformations of Tannakian categories}},
	volume = {134},
	year = {1993}}

@book{EGNO,
	author = {Etingof, Pavel and Gelaki, Shlomo and Nikshych, Dmitri and Ostrik, Victor},
	publisher = {American Mathematical Soc.},
	title = {{Tensor Categories}},
	volume = {205},
	year = {2015}}

@article{bezrukavnikov2008equivariant,
	author = {Bezrukavnikov, Roman Vladimirovich and Finkel'berg, Mikhail Vladlenovich},
	journal = {Moscow Mathematical Journal},
	number = {1},
	pages = {39--72},
	publisher = {Независимый Московский университет--МЦНМО},
	title = {{Equivariant Satake category and Kostant--Whittaker reduction}},
	volume = {8},
	year = {2008}}

@article{elliott2022taxonomy,
	author = {Elliott, Chris and Safronov, Pavel and Williams, Brian R},
	journal = {Selecta Mathematica},
	number = {4},
	pages = {73},
	publisher = {Springer},
	title = {{A taxonomy of twists of supersymmetric Yang-Mills theory}},
	volume = {28},
	year = {2022}}

@article{webster20233,
	author = {Webster, Ben and Yoo, Philsang},
	journal = {arXiv preprint arXiv:2308.06191},
	title = {{3-dimensional mirror symmetry}},
	year = {2023}}

@article{intriligator1996mirror,
	author = {Intriligator, Kenneth and Seiberg, Nathan},
	journal = {Physics Letters B},
	number = {3},
	pages = {513--519},
	publisher = {Elsevier},
	title = {{Mirror symmetry in three dimensional gauge theories}},
	volume = {387},
	year = {1996}}

@article{kapustin2007electric,
	author = {Kapustin, Anton and Witten, Edward},
	journal = {Communications in Number Theory and Physics},
	number = {1},
	pages = {1--236},
	publisher = {International Press of Boston, Inc.},
	title = {{Electric-magnetic duality and the geometric langlands program}},
	volume = {1},
	year = {2007}}

@article{bezrukavnikov2023equivariant,
	author = {Bezrukavnikov, Roman and Ionov, Andrei and Tolmachov, Kostiantyn and Varshavsky, Yakov},
	journal = {arXiv preprint arXiv:2305.02980},
	title = {{Equivariant derived category of a reductive group as a categorical center}},
	year = {2023}}

@article{gould1993quantum,
	author = {Gould, MD and Zhang, RB and Bracken, AJ},
	journal = {Bulletin of the Australian Mathematical Society},
	number = {3},
	pages = {353--375},
	publisher = {Cambridge University Press},
	title = {{Quantum double construction for graded Hopf algebras}},
	volume = {47},
	year = {1993}}

@article{beilinson1996koszul,
	author = {Beilinson, Alexander and Ginzburg, Victor and Soergel, Wolfgang},
	journal = {Journal of the American Mathematical Society},
	number = {2},
	pages = {473--527},
	title = {{Koszul duality patterns in representation theory}},
	volume = {9},
	year = {1996}}

@article{etingof1996quantization,
	author = {Etingof, Pavel and Kazhdan, David},
	journal = {Selecta Mathematica},
	pages = {1--41},
	publisher = {Springer},
	title = {{Quantization of Lie bialgebras, I}},
	volume = {2},
	year = {1996}}

@misc{pimenov,
	archiveprefix = {arXiv},
	author = {Slava Pimenov},
	eprint = {1511.00946},
	primaryclass = {math.AG},
	title = {{Shifted Poisson and Batalin-Vilkovisky structures on the derived variety of complexes}},
	url = {https://arxiv.org/abs/1511.00946},
	year = {2015}}

@misc{DN24,
	archiveprefix = {arXiv},
	author = {Tudor Dimofte and Wenjun Niu},
	eprint = {2411.04194},
	primaryclass = {hep-th},
	title = {{Tannakian QFT: from spark algebras to quantum groups}},
	url = {https://arxiv.org/abs/2411.04194},
	year = {2024}}

@article{ben2012loop,
	author = {Ben-Zvi, David and Nadler, David},
	journal = {Journal of Topology},
	number = {2},
	pages = {377--430},
	publisher = {Oxford University Press},
	title = {{Loop spaces and connections}},
	volume = {5},
	year = {2012}}

@article{ben2010integral,
	author = {Ben-Zvi, David and Francis, John and Nadler, David},
	journal = {Journal of the American Mathematical Society},
	number = {4},
	pages = {909--966},
	title = {{Integral transforms and Drinfeld centers in derived algebraic geometry}},
	volume = {23},
	year = {2010}}

@article{BDhecke,
	author = {Beilinson, Alexander and Drinfeld, Vladimir},
	journal = {preprint},
	pages = {382},
	title = {{Quantization of Hitchin's Hamiltonians and Hecke eigensheaves}},
	year = {1996}}

@misc{BGngo,
	archiveprefix = {arXiv},
	author = {David Ben-Zvi and Sam Gunningham},
	eprint = {1712.01963},
	primaryclass = {math.RT},
	title = {{Symmetries of categorical representations and the quantum Ng\^o action}},
	url = {https://arxiv.org/abs/1712.01963},
	year = {2017}}

@article{krause05,
	author = {Krause, Henning},
	journal = {Compositio Mathematica},
	number = {5},
	pages = {1128},
	publisher = {Cambridge University Press},
	title = {{The stable derived category of a Noetherian scheme}},
	volume = {141},
	year = {2005}}

@misc{Gaiindcoh,
	archiveprefix = {arXiv},
	author = {Dennis Gaitsgory},
	eprint = {1105.4857},
	primaryclass = {math.AG},
	title = {{Ind-coherent sheaves}},
	url = {https://arxiv.org/abs/1105.4857},
	year = {2012}}

@article{BFOcharacter,
	author = {Bezrukavnikov, Roman and Finkelberg, Michael and Ostrik, Victor},
	journal = {Inventiones mathematicae},
	pages = {589--620},
	publisher = {Springer},
	title = {{Character D-modules via Drinfeld center of Harish-Chandra bimodules}},
	volume = {188},
	year = {2012}}

@article{beraldo2017loop,
	author = {Beraldo, Dario},
	journal = {Advances in Mathematics},
	pages = {565--636},
	publisher = {Elsevier},
	title = {{Loop group actions on categories and Whittaker invariants}},
	volume = {322},
	year = {2017}}

@misc{NPShifted,
	archiveprefix = {arXiv},
	author = {Wenjun Niu and Victor Py},
	eprint = {2503.08770},
	primaryclass = {math.QA},
	title = {{1-shifted Lie bialgebras and their quantizations}},
	url = {https://arxiv.org/abs/2503.08770},
	year = {2025}}

@article{boyarchenko2014character,
	author = {Boyarchenko, Mitya and Drinfeld, Vladimir},
	journal = {Selecta Mathematica},
	number = {1},
	pages = {125--235},
	publisher = {Springer},
	title = {{Character sheaves on unipotent groups in positive characteristic: foundations}},
	volume = {20},
	year = {2014}}

@misc{benzvi2017characterfieldtheoryhomology,
	archiveprefix = {arXiv},
	author = {David Ben-Zvi and Sam Gunningham and David Nadler},
	eprint = {1705.04266},
	primaryclass = {math.QA},
	title = {{The Character Field Theory and Homology of Character Varieties}},
	url = {https://arxiv.org/abs/1705.04266},
	year = {2017}}

@book{milne2017algebraic,
	author = {Milne, James S},
	publisher = {Cambridge University Press},
	title = {{Algebraic groups: the theory of group schemes of finite type over a field}},
	volume = {170},
	year = {2017}}

@misc{Butsonequiv,
      title={{Equivariant localization in factorization homology and applications in mathematical physics I: Foundations}}, 
      author={Dylan Butson},
      year={2020},
      eprint={2011.14988},
      archivePrefix={arXiv},
      primaryClass={math.RT},
      url={https://arxiv.org/abs/2011.14988},}

@book{positselski2011two,
  title="{Two kinds of derived categories, Koszul duality, and comodule-contramodule correspondence}",
  author={Positselski, Leonid},
  volume={212},
  number={996},
  year={2011},
  publisher={American Mathematical Society}
}

\end{document}